\documentclass[10pt,reqno]{amsart}
\usepackage{graphicx,epsfig}
\usepackage{epstopdf}
\usepackage{pdfpages}
\usepackage{amssymb}
\usepackage{empheq}
\usepackage{cases}
\usepackage{amsthm,amsmath}
\usepackage{caption,lipsum}
\usepackage{stmaryrd}
\usepackage{tabularx}
\usepackage{color}
\usepackage{empheq,float}
\usepackage{hyperref}
\DeclareGraphicsExtensions{.eps}

\newtheorem{theorem}{Theorem}[section]
\newtheorem{lemma}[theorem]{Lemma}
\newtheorem{proposition}[theorem]{Proposition}

\theoremstyle{definition}
\newtheorem{definition}[theorem]{Definition}

\newtheorem{assum}[theorem]{Assumption}
\theoremstyle{remark}
\newtheorem{remark}[theorem]{Remark}

\newcommand{\fe}{\mathrm{e}}

\newcommand{\eps}{\varepsilon}

\newcommand{\bT}{{\mathbb T}}

\newcommand{\abs}[1]{\left\vert#1\right\vert}

\newcommand{\norm}[1]{\left\Vert#1\right\Vert}

\def\hh{h}
\numberwithin{equation}{section}

\begin{document}

\title[Two-scale integrators for Klein-Gordon equation]{Two-scale integrators with high accuracy and long-time conservations for the nonlinear Klein-Gordon equation in the nonrelativistic limit regime}

\author[B. Wang]{Bin Wang}\address{\hspace*{-12pt}B.~Wang: School of Mathematics and Statistics, Xi'an Jiaotong University, 710049 Xi'an, China}
\email{wangbinmaths@xjtu.edu.cn}

\author[Z. Miao]{Zhen Miao}
\address{\hspace*{-12pt}  Z. Miao: School of Mathematics and Statistics,
          Northwestern Polytechnical University,
       710072  Xi'an,  China}
\email{mz91127@126.com}

\author[Y. L. Jiang]{Yaolin Jiang}
\address{\hspace*{-12pt} Y. L.~Jiang: School of Mathematics and Statistics, Xi'an Jiaotong University, 710049 Xi'an, China}
\email{yljiang@mail.xjtu.edu.cn}
\date{}

\dedicatory{}

\begin{abstract}
In this paper, we are concerned with two-scale integrators  for the non-relativistic Klein--Gordon (NRKG) equation with a dimensionless parameter $0<\varepsilon\ll 1$, which is inversely proportional to the speed of light. The highly oscillatory property in time of this model corresponds to the parameter $\varepsilon$ and the equation {in the form of $\partial_{tt}u -\frac{\Delta}{\eps^2}  u  +\frac{1}{\eps^4}u +\frac{\lambda}{\varepsilon^2} f(u)=0$} {has a factor $1/\varepsilon^2$ in front of the  nonlinearity which means that this part becomes strong when $\varepsilon$ is small. These} two aspects bring significantly numerical burdens in designing numerical methods. {We propose  a  class of two-scale integrators which is constructed based on some reformulations to the system, Fourier pseudo-spectral method and exponential integrators.} Two practical integrators up to order three and four are constructed by using some symmetric conditions and the stiff order conditions of implicit exponential integrators. The convergence of the obtained integrators is rigorously studied, and it is shown that the uniform accuracy in time is $\mathcal{O}(h^3)$ and $\mathcal{O}(h^4)$ for the time stepsize $h$. The near energy conservation over long times is also established for the multi-stage integrators by using modulated Fourier expansions. 
Numerical results on a NRKG equation show that the proposed integrators  have high accuracy, excellent long time energy conservation and competitive efficiency.\\
{\bf Keywords:} Two-scale integrators, High accuracy, Long time near conservation, Nonlinear Klein-Gordon equation, Modulated Fourier expansion. \\
{\bf AMS Subject Classification:} 65M12, 65M15, 65M70.
\end{abstract}

\maketitle

\section{Introduction}
{A significant portion of dynamical systems can be characterized by the non-relativistic Klein-Gordon (NRKG) equation, which is expressed as follows}  (\cite{Bao21,zhao19,PI1,Chartier15,zhao18})
\begin{equation}\label{klein-gordon}
\begin{array}
[c]{l}%
 \eps^2 \partial_{tt}u(x,t) -\Delta u(x,t) +\frac{1}{\eps^2}u(x,t)+\lambda f\big(u(x,t)\big)=0,\
\ \  x\in \Omega_x^d,\ \ t\in[0,T],\\
u(x,0)=\psi_1(x),\ \
 u_{t}(x,0)=\frac{1}{\eps^2}\psi_2(x),
\end{array}
\end{equation}
where   $t$ is time, $x\in \Omega_x^d$ is the spatial coordinate with  $d$-dimensional ($d\geq1$) space torus $\Omega_x^d$, $\Delta$ is the Laplacian operator, $u(x,t)$ is a complex-valued scalar field, $0<\eps \ll 1$
is a dimensionless parameter inversely proportional to the speed of light, $\lambda \in \mathbb{R}$ is a given dimensionless
parameter (positive and negative for defocusing and focusing self-interaction, respectively), $ f(u)$ is a nonlinear {function  $f(u)=\nabla H_1(u)$ with a smooth
potential $H_1$, and} $\psi_1$ and $\psi_2$
are given complex-valued $\eps$-independent {functions}.
 The NRKG \eqref{klein-gordon} is time symmetric and conserves the energy
\begin{equation}
H(u,v)=\int_{\Omega_x^d}\Big(\eps^2\abs{v}^2+
\abs{\nabla u}^2+\frac{1}{\eps^2}\abs{u}^2+\lambda H_1(u)\Big)dx\equiv H(\psi_1, \psi_2/\eps^2),
 \label{hhh}
\end{equation}
with $v:=\partial_{t}u$. It is noted that the scales of $\psi_1,
\psi_2$ {lead to}  different properties of the
system.  In this paper,  {we consider the  scale $\psi_1=\mathcal{O}(1),  \psi_2=\mathcal{O}(1)$ and this gives an
energy unbounded system \eqref{klein-gordon} with large initial data.
Although such system  can be converted into an energy-bounded system through a scaling of the variables, the case with large initial data exhibits more fundamental distinctions, such as more pronounced oscillatory characteristics. This nuance renders the development and analysis of numerical methods for energy unbounded system   more difficult and challenging (\cite{Zhao}).}
 We should note that all the methods and analysis presented in this paper are applicable to energy bounded system \eqref{klein-gordon} where a small $\psi_2$ is chosen as $\psi_2=\mathcal{O}(\eps^2)$.  {This study employs periodic boundary conditions for clarity of presentation. However, the extension to other boundary types is trivial, as the theoretical framework is centered on the time integrator.}

 {The model \eqref{klein-gordon}} often
arises in a variety of fields such as plasma physics  for modeling interaction between Langmuir and ion sound waves  and   cosmology as a phonological model for dark-matter and/or black-hole evaporation.
Its computation represents
  major challenges because of two aspects. a) The solution of
\eqref{klein-gordon} becomes highly oscillatory in time when large
frequencies are involved in the equation, which corresponds to the
parameter $0 < \eps\ll  1$. b) {The nonlinear term $f(u)$ is preceded by a factor of $1/\varepsilon^2$ (divide both sides of the equation \eqref{klein-gordon} by $\eps^2$),
  indicating that this component} is very strong when $0 < \eps\ll  1$. {This aspect} also brings  great challenges for effective scientific computing.

The time integration of such equation \eqref{klein-gordon} is a basic
algorithmic task and it has been received much attention in recent decades.  Due to the high oscillations and {strong part $f(u)/\eps^2$,}
 the  traditional methods such as symplectic/symmetric Runge--Kutta--Nystr\"{o}m (RKN) methods (\cite{Feng-book,Sanz-Serna}) or energy-preserving Runge-Kutta methods (\cite{r4,Shen19}) often {result in} convergence problems.
In order to get more competitive methods for second-order systems, Gautschi-type trigonometric
integrators were developed in
\cite{Hochbruck1999} and different kinds of trigonometric
integrators were formulated and  studied in
\cite{franco2006,MD,Grimm,hairer2000,WIW,Wu}. However, for these
integrators, they at most have  uniform second order accuracy in the
absolute position error w.r.t. $\eps $ and only have uniform first
order accuracy in velocity error \cite{Grimm} for not strong   nonlinearity. {As a
result, it is} challenging to get trigonometric integrators with
{high-order   accuracy} for the nonlinear Klein-Gordon equation in the nonrelativistic limit regime.

Recently, some new methods with uniform accuracy (w.r.t $\eps$) for  highly  oscillatory systems have been proposed and analysed such as
  two-scale formulation   methods  \cite{Chartier15,vp3D}, uniformly accurate exponential-type integrators \cite{PI1,S2},  nested Picard iterative integrators \cite{PI2}, multiscale time integrators \cite{Baocai,bao14,2scale multi},   multi-revolution composition
methods \cite{zhao18}, uniformly accurate methods with averaging \cite{NUA,autoUA} and low regularity integrators \cite{S1}.  In \cite{zhao19,Schratz},  various uniformly accurate  (UA) methods  and asymptotic expansion techniques have been  compared systematically for solving the  NRKG equation \eqref{klein-gordon}.    Recently,  time-splitting methods were proved to have
 improved uniform error bounds for solving nonlinear Klein-Gordon equation with weak nonlinearity \cite{Bao21}.
Most of these uniformly accurate methods can be applied to  the system \eqref{klein-gordon} and they have   uniform accuracy in both position and velocity (\cite{zhao19}).  Unfortunately,
  most of them only have up to second uniform accuracy and do not have good long time conservation behaviour when applied to conservative systems. More precisely,
if a UA method is considered as the approximation of
\eqref{klein-gordon}, the numerical energy error will increase
as  time evolves. In a recent work \cite{NUA}, the UA method named
as pullback method is shown to hold the long time conservation  by
the numerical results but without rigorous analysis. In a more
recent work \cite{WZ21}, the authors succeed in making the two-scale
method with near conservation laws for first-order systems. However,
only one-stage type  methods of order two are presented there and the equation
\eqref{klein-gordon} does not share the form of the system
considered in \cite{WZ21}, which means that the analysis of
\cite{WZ21} is no longer applicable for the system
\eqref{klein-gordon} of this paper. Moreover, the second-order
differential  equation \eqref{klein-gordon} has its special
structure which will be neglected if we rewrite it as a general
highly oscillatory first-order differential equation.  Therefore, it is necessary and
meaningful to design and analyze UA methods {with uniform high order accuracy
and  good long time energy near} conservation for solving the
second-order system \eqref{klein-gordon}. It is worth pointing
out that  the integrators derived in this paper will be shown to
have long time {near conservation and the uniform} accuracy  $\mathcal{O}(h^3)$ or $\mathcal{O}(h^4)$ for the time {stepsize} $h$.
 {This high accuracy} is different from  the existing UA methods \cite{Baocai,bao14,zhao19,PI1,PI2,Chartier15,vp3D,2scale multi,NUA,autoUA,zhao18,WZ21} because   {previous researches have primarily focused on second-order UA algorithms. }

In this paper we are interested in using numerical integrators with
time {stepsizes}  that are much larger than the $\eps$
of the system to obtain high accuracy  and good long-time
energy conservation.  {This work addresses both high-order UA algorithms and their nice long-term behaviour, with the main challenge—compared to \cite{Chartier15}—being the need to achieve these two goals concurrently.} {To
this end}, we take advantage of two-scale
formulation approach, spectral semi-discretisation and
exponential integrators with more than one stage. However, this brings some challenges and
difficulties in the achievement and analysis of long time energy
conservation. a) The two-scale formulation approach results in a new
system which has a completely different  linear part and nonlinear
function in comparison  with the original problem
\eqref{klein-gordon}. {This alteration introduces additional challenges in establishing the proof of long time near energy conservation for the initial problem
\eqref{klein-gordon}.} 
 b) To
{achieve}  high order  accuracy and good
long time behaviour, two-stage and three-stage exponential integrators are chosen in
this paper. {Unfortunately, however,} for a method with
more than one stage applied to highly oscillatory systems, {it seems to us that long time
analysis has not been done.} As pointed out in
\cite{Cohen0}, there is the technical difficulty coming from the
identification of invariants in the corresponding modulation system.
Thus, it remains a challenge to study long time behaviour for a
method with more than one stage. c) The third challenge comes
{from} the diversity of \eqref{klein-gordon}
considered in this paper. Large initial data case
usually leads to large bounds of the coefficient functions in the
modulated Fourier expansion, which prevents the derivation of the
long time conservation.

To overcome these difficulties and make the analysis go smoothly,
a novel approach to the design of  integrators is established. The main  contributions of this paper are as follows.

a) We  consider some transformations of the original system and
use two-scale exponential integrators which  satisfy the derived stiff
order conditions. The transformations of the system and stiff order conditions
proposed in this paper can keep the high accuracy, and
the symmetric conditions used in the formulation of the methods yield  good long time near energy
conservation. It will be shown that these both  important properties can be hold for the integrators used with large time stepsizes, which is very  effective and efficient in scientific computing over long times.

  b) Compared with the existing UA methods, the accuracy of our   integrators is proved to be $\mathcal{O}(h^3)$ and $\mathcal{O}(h^4)$ for solving \eqref{klein-gordon}, where $h$ is the time stepsize. This  high accuracy is   very competitive  in the numerical computation of  NRKG  equation where $0<\eps\ll 1$.

 c) Moreover,   we should note that we managed to derive the long-time energy  conservation  for the two-stage and three-stage methods applied to the large initial value system, which is different form  the existing long term analysis work \cite{Cohen0,Cohen2,Lubich2006,Lubich2,hairer2000,lubich19,WW,WZ21},  where only one-stage type methods and small initial value are both necessary. The long time analysis presented in this paper provides an extension of the powerful  technique named as modulated Fourier expansion  \cite{Cohen0,Lubich2006,Lubich2,hairer2000,lubich19} to multi-stage methods and with such extension it is believed that more numerical methods with complicated scheme can be studied.


The rest of this paper is organized as follows. In Section \ref{sec:2}, we firstly present the formulation process of the integrators, and then we construct two practical integrators by using the proposed symmetry and stiff order conditions.  The main results concerning the error bound and near energy conservation are given in Section \ref{sec:3}, and some numerical resuts  are made  to show these two properties. The convergence is proved in Section \ref{sec:4}, and  the long-time analysis of energy conservation is drawn in Section \ref{sec:5}. The last section includes the conclusions of this paper.

\section{Formulation of the numerical integrators}\label{sec:2}

\subsection{The construction process of integrators}
In the formulation of the numerical scheme, we
first make some transformations of the system and then consider the numerical integration. There are in all three steps in the process and we present them one by one.

\textbf{Step 1. Some transformations of the system.} Firstly, we rewrite   the original system \eqref{klein-gordon}  as
 \begin{equation}\label{2ode}
\begin{array}[c]{ll}
 \partial_{tt}u(x,t) +\frac{1-\eps^2\Delta}{\eps^4}u(x,t) +\frac{1}{\eps^2} \lambda f\big(u(x,t)\big)=0.
\end{array}
\end{equation}
{By using a scaling to the variables $$
 \tilde{u}(x,t):=u(x,t)\ \  \textmd{and} \ \ \tilde{v}(x,t):=\eps^2  \partial_{t}u(x,t) ,$$
a new system is obtained immediately
\begin{equation}\label{H model problem}
\begin{aligned} &  \partial_{t} \tilde{u}(x,t)=\frac{1}{\eps^2}\tilde{v}(x,t), \qquad \qquad  \qquad \qquad \qquad\check{u}(x,0)= \psi_1(x),\\
 & \partial_{t} \tilde{v}(x,t)=-\frac{1-\eps^2\Delta}{\eps^2}\tilde{u}(x,t)-\lambda f(\check{u}(x,t)), \quad \check{v}(x,0)= \psi_2(x),\ \ t\in[0,T].
\end{aligned}
\end{equation}}Then letting
 \begin{equation*}
\begin{aligned} &\tilde{u}(x, t ):= \big(\sqrt{1-\eps^2\Delta}\big)^{-1}\Big[\cos\Big( \frac{ t \sqrt{1-\eps^2\Delta}}{\eps^2}\Big) q(x, t )+  \sin\Big( \frac{ t \sqrt{1-\eps^2\Delta}}{\eps^2}\Big)  p(x, t )\Big],\\
 &\tilde{v}(x, t ):=-  \sin\Big( \frac{ t \sqrt{1-\eps^2\Delta}}{\eps^2}\Big) q(x, t )+\cos\Big( \frac{ t \sqrt{1-\eps^2\Delta}}{\eps^2}\Big) p(x, t ), \end{aligned}
\end{equation*}
 we {obtain}
  \begin{equation}\label{new-system}%
\begin{aligned}
& \partial_{ t }q(x, t )=-   \sin\Big( \frac{ t \sqrt{1-\eps^2\Delta}}{\eps^2}\Big){g(t,q(x, t ),p(x, t ))},\qquad q(x,0)= \sqrt{1-\eps^2\Delta}\psi_1(x),\\
& \partial_{ t }p(x, t )=  \cos\Big(\frac{ t \sqrt{1-\eps^2\Delta}}{\eps^2}\Big){g(t,q(x, t ),p(x, t ))},\qquad \ \ \
p(x,0)= \psi_2(x),
 \end{aligned}
\end{equation}
where $${g(t,q,p)=-\lambda f\Big(\big(\sqrt{1-\eps^2\Delta}\big)^{-1}\Big[\cos\Big( \frac{ t \sqrt{1-\eps^2\Delta}}{\eps^2}\Big) q+  \sin\Big( \frac{ t \sqrt{1-\eps^2\Delta}}{\eps^2}\Big)  p\Big]\Big)}.$$

{For   the propagators $ \sin\Big( \frac{ t \sqrt{1-\eps^2\Delta}}{\eps^2}\Big)$  and $ \cos\Big( \frac{ t \sqrt{1-\eps^2\Delta}}{\eps^2}\Big)$,  we deal with them as
  \begin{equation}\label{scf}%
\begin{aligned}
& \sin\Big( \frac{ t \sqrt{1-\eps^2\Delta}}{\eps^2}\Big)=\sin\big( t/\eps^2+t D_{\eps}\big)
=\sin(t/\eps^2)\cos(t D_{\eps})+\cos(t/\eps^2)\sin(t D_{\eps}),\\
& \cos\Big( \frac{ t \sqrt{1-\eps^2\Delta}}{\eps^2}\Big)=\cos\big( t/\eps^2+t D_{\eps}\big)
=\cos(t/\eps^2)\cos(t D_{\eps})-\sin(t/\eps^2)\sin(t D_{\eps})
 \end{aligned}
\end{equation}
with an operator $D_{\eps}:=\frac{  \sqrt{1-\eps^2\Delta}-1}{\eps^2}: H^{\nu+2}\rightarrow H^{\nu}$ which is   uniformly  bounded  w.r.t. $\eps$. From \eqref{scf}, it is clear that
 the propagators $ \sin\Big( \frac{ t \sqrt{1-\eps^2\Delta}}{\eps^2}\Big)$  and $ \cos\Big( \frac{ t \sqrt{1-\eps^2\Delta}}{\eps^2}\Big)$  can be expressed by the  $2\pi$-periodic functions $\sin(t/\eps^2)$ and  $\cos(t/\eps^2)$ in $ t /\eps^2$.}

\textbf{Step 2. Two-scale formulation.} By isolating the fast time variable $ t /\eps^2$ as another variable $\tau$ and denoting
 $$U(x, t , t /\eps^2):=q(x, t ),\ V(x, t , t /\eps^2):=p(x, t ),$$
the two-scale {pattern of (\ref{new-system}) can be
formulated as follows:}
\begin{equation}\label{2scale}\begin{split}
  &\partial_{ t }U(x, t ,\tau)+\frac{1}{\eps^2}\partial_\tau U(x, t ,\tau)\\
  =&-    \Big(\sin(\tau)\cos(t D_{\eps})+\cos(\tau)\sin(t D_{\eps})\Big)F(t,\tau,U(x, t ,\tau),V(x, t ,\tau)),\\
    & \partial_{ t }V(x, t ,\tau)+\frac{1}{\eps^2}\partial_\tau V(x, t ,\tau)\\
     =&\Big(\cos(\tau)\cos(t D_{\eps})-\sin(\tau)\sin(t D_{\eps})\Big)F(t,\tau,U(x, t ,\tau),V(x, t ,\tau)),\end{split}
\end{equation}
where $ t \in[0,T],\ \tau\in\bT,$ $U(x, t ,\tau)$ and $V(x, t ,\tau)$ are the unknowns which are periodic in $\tau$ on the
torus $\mathbb{T}=\mathbb{R}/2\pi\mathbb{Z}$ and
\begin{equation*} \begin{split}{F(t,\tau,U,V)}=&-\lambda f\Big(\big(\sqrt{1-\eps^2\Delta}\big)^{-1}\Big[\Big(\cos(\tau)\cos(t D_{\eps})-\sin(\tau)\sin(t D_{\eps})\Big)U\\
&+\Big(\sin(\tau)\cos(t D_{\eps})+\cos(\tau)\sin(t D_{\eps})\Big) V\Big]\Big).\end{split}
\end{equation*}
Letting{
\begin{equation*} \begin{split}&G(t,\tau,X(x, t ,\tau))\\
 :=&\left(
                      \begin{array}{c}
                      -\Big(\sin(\tau)\cos(t D_{\eps})+\cos(\tau)\sin(t D_{\eps})\Big)F(t,\tau,U(x, t ,\tau),V(x, t ,\tau))\\
                       \Big(\cos(\tau)\cos(t D_{\eps})-\sin(\tau)\sin(t D_{\eps})\Big)F(t,\tau,U(x, t ,\tau),V(x, t ,\tau)) \\
                      \end{array}
                    \right)\end{split}
\end{equation*}}with $X:=[U;V],$
 the above system can be rewritten
in a compact form \begin{equation}\label{2scale compact} \partial_{ t } X(x, t ,\tau)+\frac{1}{\eps^2}
 \partial_\tau X(x, t ,\tau)=G(t,\tau,X(x, t ,\tau)),\ \   t \in[0,T],\ \tau\in\bT.\end{equation}

For this new system, the following assumption is required in this paper.
\begin{assum}\label{ass}
Consider that the unknowns of  \eqref{klein-gordon}   $t \rightarrow u$  and $t \rightarrow \partial_t u$ are   maps  onto the Sobolev spaces  $H^{\nu+1}(\Omega_x^d)$  and  $H^{\nu}(\Omega_x^d)$ with  {integers} $\nu\geq0$ and
the $d$-dimensional space torus $\Omega_x^d$ ($d\geq 1$), respectively. {It is also required that these maps have continuous derivatives (w.r.t. $t$ and $x$) up to fourth order.}
In addition,
 it is required that   $\nu\geq 4+d$ which is necessary for the analysis in Banach algebras. For all $\alpha,\beta,\gamma \in \{0,1,\ldots,4\}$  and {for the function $G(t,\tau,X),$ it is assumed that the functional  $\partial^{\alpha}_{\tau}\partial^{\beta}_{t}\partial^{\gamma}_{X}G(t,\tau,X)$ is continuous and locally
bounded w.r.t. $\eps$ from   $\bT\times H^{\nu}$  to $\mathcal{L}(\underbrace{H^{\sigma}\times \cdots\times H^{\sigma}}_{\beta\ \textmd{times}},H^{\sigma-\beta})$, where   $\nu\geq\sigma\geq\beta+d$.}
\end{assum}

\begin{remark}
It is noted that this new variable $\tau$ of \eqref{2scale compact} offers a free degree for designing the initial date $X(x,0,\tau)$.
This two-scale equation \eqref{2scale} with a modified initial data is analysed in  \cite{Chartier15,autoUA} to construct uniformly accurate methods.
We here use the strategy from
\cite{Chartier15,autoUA} to obtain the fourth-order  initial data for
(\ref{2scale}), which is presented briefly as follows.
\end{remark}
 {For the periodic function  $v(\cdot)$ on $\bT$, we introduce the  notations
 $$L:=\partial_\tau,\ \ \ \
 \Pi v:=\frac{1}{2\pi}\int_0^{2\pi}v(\tau)d\tau.$$
 On the set of functions with vanishing average, $\partial_{\tau}$
is invertible with inverse defined by (see \cite{Chartier15})
$L^{-1} v:=(I-\Pi) \int_0^{\tau }v(\theta)d\theta$. By further letting $A:=L^{-1}(I-\Pi)$, it is easy to get
}
   \begin{equation}\label{pro ope}
\begin{aligned}
& L \Pi=0,\ \ \Pi^2=\Pi,  \ \ \Pi L v=0,\ \ \Pi A=0,\\
&
LA=I-\Pi,\ \ LA^2=(I-\Pi) A=A-\Pi  A=A.\end{aligned}
\end{equation}
Based on the Chapman-Enskog expansion   (\cite{NUA,autoUA}), the solution $X(x, t ,\tau)$ {of \eqref{2scale compact}} can be formulated as
   $$X(x, t ,\tau)=\underline{X}(x, t )+\kappa(x, t ,\tau)\ \ \textmd{with}\ \ \underline{X}(x, t )=\Pi X(x, t ,\tau),\ \ \Pi\kappa(x, t ,\tau)=0.$$
 From the fact   \eqref{2scale compact},
  it follows that $$\partial_{ t } \underline{X}(x, t )=\Pi  G(t,\tau,\underline{X}(x, t )+\kappa(x, t ,\tau)).$$ Then the   correction $\kappa$ satisfies the differential equation
   $$\partial_{ t } \kappa(x, t ,\tau)+\frac{1}{\eps^2}
 \partial_\tau \kappa(x, t ,\tau)=(I-\Pi)G(t,\tau,\underline{X}(x, t )+\kappa(x, t ,\tau)),$$
 which can be reformulated as
  \begin{equation}\label{kapa de}
\kappa(x, t ,\tau)=\eps^2 A G(t,\tau,\underline{X}(x, t )+\kappa(x, t ,\tau))-\eps^2 L^{-1}(\partial_{ t }\kappa(x, t ,\tau)).
\end{equation}
 We seek an expansion in powers of $\eps$ for the composer $\kappa$:
 \begin{equation}\label{kapa}
 {\kappa(x, t ,\tau)=\eps^2 \kappa_1(\tau,\underline{X}(x, t ))+\eps^4 \kappa_2(\tau,\underline{X}(x, t ))+\eps^6 \kappa_3(\tau,\underline{X}(x, t ))+\mathcal{O}_{\mathcal{X}^{\nu}}(\eps^{8})},
\end{equation}
where $\mathcal{O}_{\mathcal{X}^{\nu}}$ denotes the term uniformly bounded with the $\mathcal{X}^{\nu}$-norm
and $\mathcal{X}^{\nu}:=\bigcap_{0\leq l\leq\nu-d}C^l(\mathbb{T};H^{\nu-l}).$
 Inserting \eqref{kapa} into \eqref{kapa de} and using the Taylor series of $G$ at $\underline{X}$ yield
 \[
\begin{aligned}
\eps^2 \kappa_1&+\eps^4 \kappa_2+\eps^6 \kappa_3+\mathcal{O}_{\mathcal{X}^{\nu}}(\eps^{8})=
\eps^2 A G(t,\tau,\underline{X})+\eps^2 A \partial_{X}G(t,\tau,\underline{X})\big(\eps^2 \kappa_1+\eps^4 \kappa_2\big)\\&+\frac{1}{2}\eps^2 A \partial^2_UG(t,\tau,\underline{X})\big(\eps^2 \kappa_1,\eps^2 \kappa_1\big)-\eps^2 L^{-1}(\eps^2 \partial_{ t }\kappa_1+\eps^4 \partial_{ t }\kappa_2)+\mathcal{O}_{\mathcal{X}^{\nu}}(\eps^{8}).\end{aligned}
\]
 By comparing the coefficients   of $\eps^j$ with $j=1,2,3$, one gets
  \begin{equation}\label{kapa re}\begin{aligned}
&
\kappa_1=AG(t,\tau,\underline{X}),\ \ \kappa_2=A\partial_{X}G(t,\tau,\underline{X})AG(t,\tau,\underline{X})-L^{-1}  \partial_{ t }\kappa_1,\\
 &\kappa_3=A\partial_{X}G(t,\tau,\underline{X})\kappa_2+\frac{1}{2}  A \partial^2_UG(t,\tau,\underline{X})\big(  \kappa_1, \kappa_1\big)-L^{-1}  \partial_{ t }\kappa_2.
\end{aligned}
\end{equation}

 With these preparations, consider the following initial condition
 \begin{equation}\label{inv}X_0(\tau):=\underline{X}^{[0]}+\eps^2 \kappa_1(\tau,\underline{X}^{[0]})+\eps^4 \kappa_2(\tau,\underline{X}^{[0]})+\eps^6 \kappa_3(\tau,\underline{X}^{[0]}),\end{equation}
 where $\underline{X}^{[0]}$ is given by $\underline{X}^{[0]}:=\Pi [q(x,0);p(x,0)]$. Observing that $\Pi,\ A$ are bounded on $C^0(\mathbb{T};H^{\sigma})$ for any $\sigma$, it is deduced that
$\kappa_1,\kappa_1, \kappa_3$ all belong to $\mathcal{X}^{\nu}$.
For the two-scale system \eqref{2scale} with  {the   initial data \eqref{inv}},  the following   estimate of the solution is obtained and its proof will be given in the next section.

\begin{proposition} \label{TS M} 
Under the conditions given in  Assumption \ref{ass}, it is clear that   $\underline{X}^{[0]}$ is   uniformly bounded in $\eps$ w.r.t.  the $H^{\nu}$ norm.
 Then,  the two-scale system \eqref{2scale compact} with the  initial condition \eqref{inv} has a unique solution
$X(\cdot, t ,\tau)\in C^0([0,T]\times\bT;H^{\nu})$ with the bound
$$\norm{X(\cdot, t ,\tau)}_{H^{\nu}}\leq
C_1 \norm{X_0(\tau)}_{H^{\nu}} \leq
C_2 \  \forall\       t  \in(0,T],$$
where the constants $C_1,C_2 > 0$ are independent of $\eps$ but depend on $T$.
Moreover, the   derivatives of $X(\cdot, t ,\tau)$ are bounded by
 {$$
\norm{
\partial ^{\mu}_{ t }X(\cdot, t ,\tau)}_{L^{\infty}_{\tau}(H^{\nu-\mu})}\leq C_3  \forall\       t  \in(0,T],$$}where $\mu=1,2,3,4$ and $C_3$ is independent of $\eps$ but depends on $T$.
\end{proposition}


\textbf{Step 3. Fully discrete scheme.}
This step is devoted to the fully discrete scheme.

 Firstly, for the problem  \eqref{klein-gordon}, from the point of practical computations, the operator $-\Delta$  is in general approximated by a symmetric and positive semi-definite differential matrix.
 In this paper, we consider the Fourier spectral
collocation (FSC) (see, \cite{Shen}).
Then a
discrete  two-scale system for \eqref{2scale compact}   is obtained
 \begin{equation}\label{dis-2scale compact} \partial_{ t }  Z( t ,\tau)+\frac{1}{\eps^2}
 \partial_\tau  Z( t ,\tau)=\widehat{G} (t,\tau,Z ( t ,\tau)),\ \   t \in[0,T],\ \tau\in\bT.\end{equation}
Here
 $\widehat{G}_{\tau}(t,\tau,Z ( t ,\tau))$ with $Z :=[Z^U;Z^V]$ is the approximation of $G(t,\tau,X)$ by the Fourier spectral
collocation, where $-\Delta$ is replaced by a
$\hat{d}$-dimensional matric $A$, $Z^U$ and $Z^V$ are $\hat{d}$-dimensional vectors which are respectively  referred to the discrete $U$ and $V$ in $x$,
and the initial value of \eqref{dis-2scale compact} is
 \begin{equation}\label{dis-inv}Z _0(\tau):=\underline{Z }^{[0]}+\eps^2 \kappa_1(\tau,\underline{Z }^{[0]})+\eps^4 \kappa_2(\tau,\underline{Z }^{[0]})+\eps^6 \kappa_3(\tau,\underline{Z }^{[0]}),\end{equation}
 where  $\underline{Z }^{[0]}:=\Pi [\hat{q}(0);\hat{p}(0)]$  and $\hat{q}(0)$ or $\hat{p}(0)$ denotes  the discrete $q(x,0)$ or $p(x,0)$ in $x$, respectively. By the analysis presented in \cite{Shen}, it is obtained immediately that
$$\norm{X(x, t ,\tau)-Z( t ,\tau)}_{H^{\nu}}\leq  \delta^x_{\mathcal{F}},$$
 where
 $\delta^x_{\mathcal{F}}$  denotes {the error  brought by the Fourier pseudo-spectral method in $x$.}
%


We now present the fully discrete scheme for \eqref{2scale} which is constructed by making use of the spectral semi-discretisation  (see
\cite{Shen})  in $\tau$ and exponential integrators (see
\cite{Ostermann})  in time.
%

 For the spectral semi-discretisation  used in $\tau$, let  $\mathcal{M}:=\{-N_\tau/2,-N_\tau/2+1,\ldots,N_\tau/2\}$ with a positive integer
 $N_\tau>1$ and
$Y_{\mathcal{M}}=\textmd{span}\{e^{ \mathrm{i}   k \tau },\ k\in \mathcal{M},\ \tau \in[-\pi,\pi]\}.
$
For any periodic function $v(\tau)$ on $[-\pi, \pi]$, define the standard projection operator  $P_{\mathcal{M}} : L^2([-\pi, \pi]) \rightarrow
Y_{\mathcal{M}}$ and the trigonometric interpolation operator $I_{\mathcal{M}} : C([-\pi, \pi]) \rightarrow Y_{\mathcal{M}}$
respectively as
\begin{equation}\label{PI}(P_{\mathcal{M}}v)(\tau)=\sum\limits_{k\in \mathcal{M}}\widetilde{v_k} e^{ \mathrm{i}k \tau },\ \
(I_{\mathcal{M}}v)(\tau)=\sum\limits_{k\in \mathcal{M}}
\widehat{v_k} e^{ \mathrm{i}k \tau },\end{equation} where $\mathrm{i}=\sqrt{-1}$,  $\widetilde{v}_k$
{for $k\in \mathcal{M}$} are the Fourier transform
coefficients of the periodic function $v(\tau)$  and $\widehat{v}_k$
are the discrete Fourier transform coefficients of the vector
$\{v(\tau_k)\}_{\tau_k:=\frac{2 \pi}{N_\tau}k}$. For consistency
reasons, we assume that the first term and the last one in the
summation are taken with a factor $1/2$ here and after. Then the
Fourier spectral method is given by finding the trigonometric
polynomial
$$Z^{\mathcal{M}}( t ,\tau)=\big(Z_l^{\mathcal{M}}( t ,\tau)\big)_{l=1,2,\ldots,2\hat{d}},\  ( t ,\tau)\in[0,T]\times
[-\pi,\pi]$$
with $$
 Z_l^{\mathcal{M}}( t ,\tau)=\sum\limits_{k\in \mathcal{M}}
 \widetilde{Z_{k,l}( t )}\mathrm{e}^{\mathrm{i} k \tau }$$
such that
$$\partial_{ t } Z^{\mathcal{M}}( t ,\tau)+\frac{1}{\eps^2}\partial_\tau Z^{\mathcal{M}}( t ,\tau)= \widehat{G}(t,\tau,Z^{\mathcal{M}}( t ,\tau)).$$
{It follows from the orthogonality of the Fourier
functions and collecting all the $\widetilde{Z_{k,l}}$ in   $(N_\tau+1)$-periodic coefficient vectors $
\widetilde{\mathbf{Z}( t ) }= (\widetilde{Z_{k,l}( t ) )}$ that}
\begin{equation}\label{2scale Fourier}\begin{aligned}
&\frac{d}{d  t }\widetilde{\mathbf{Z}( t ) }=\mathrm{i}\Omega\widetilde{\mathbf{Z}( t ) }
+ \Gamma \big(t,\widetilde{\mathbf{Z}( t ) }\big),
 \end{aligned}
\end{equation}
%
where $\Gamma \big(t,\widetilde{\mathbf{Z}}\big):=-\lambda\left(
                      \begin{array}{c}
                       - \mathcal{F}\mathbf{S}(t)f\big(\mathcal{C}(t)\widetilde{\mathbf{Z}}^1+\mathcal{S}(t)\widetilde{\mathbf{Z}}^2\big)\\
                         \mathcal{F}\mathbf{C}(t)f\big(\mathcal{C}(t)\widetilde{\mathbf{Z}}^1+\mathcal{S}(t)\widetilde{\mathbf{Z}}^2\big) \\
                      \end{array}
                    \right) $, the vector $\widetilde{\mathbf{Z}}:=[\widetilde{\mathbf{Z}}^1;\widetilde{\mathbf{Z}}^2]$ is in dimension $D:=2\hat{d}\times (N_{\tau}+1)$ with its block vectors $\widetilde{\mathbf{Z}}^1,\widetilde{\mathbf{Z}}^2$ in dimension $D/2$,   $\mathcal{F}$
denotes the discrete Fourier transform, 
$\Omega:=\textmd{diag}(\Omega_1,\Omega_2,\ldots,\Omega_{2\hat{d}})$ with
$\Omega_1=\Omega_2=\cdots=\Omega_{2\hat{d}}:=\frac{1}{\eps^2}\textmd{diag}\big(\frac{N_{\tau}}{2},
  \frac{N_{\tau}}{2}-1,\ldots,-\frac{N_{\tau}}{2}\big)$, and
    \begin{equation*}  \begin{array}
[c]{l}%
\mathbf{S}(t):= \cos(t D_{A})\otimes \sin(\Omega_1)+ \sin(t D_{A})\otimes \cos(\Omega_1),\\
    \mathbf{C}(t):= \cos(t D_{A})\otimes \cos(\Omega_1)- \sin(t D_{A})\otimes \sin(\Omega_1),\\
\mathcal{S}(t):=  \frac{\cos(t D_{A})}{ \sqrt{1+\eps^2 A}}\otimes \sin(\Omega_1)+  \frac{\sin(t D_{A})}{ \sqrt{1+\eps^2 A}}\otimes \cos(\Omega_1),\\
\mathcal{C}(t):= \frac{\cos(t D_{A})}{ \sqrt{1+\eps^2 A}}\otimes \cos(\Omega_1)- \frac{\sin(t D_{A})}{ \sqrt{1+\eps^2 A}}\otimes \sin(\Omega_1),
\end{array}
\end{equation*}
with $D_{A}=\frac{  \sqrt{1+\eps^2 A}-I}{\eps^2}$.
 Our analysis {presented below} will use the {entries} of  $\widetilde{\mathbf{Z}}$, which are  denoted by
  \begin{equation*} \begin{aligned}\widetilde{\mathbf{Z}}=\big(&\widetilde{Z_{-\frac{N_{\tau}}{2},1}},\widetilde{Z_{-\frac{N_{\tau}}{2}+1,1}},\ldots,
 \widetilde{Z_{\frac{N_{\tau}}{2},1}},\widetilde{Z_{-\frac{N_{\tau}}{2},2}},\widetilde{Z_{-\frac{N_{\tau}}{2}+1,2}},\\
 &\ldots,
 \widetilde{Z_{\frac{N_{\tau}}{2},2}},\ldots,\widetilde{Z_{-\frac{N_{\tau}}{2},2\hat{d}}},
 \widetilde{Z_{-\frac{N_{\tau}}{2}+1,2\hat{d}}},\ldots,
 \widetilde{Z_{\frac{N_{\tau}}{2},2\hat{d}}}\big).\end{aligned}
\end{equation*}
The same     notation
 is used   for all the
vectors and diagonal matrices with the same dimension as $\widetilde{\mathbf{Z}}$.
 We also use the notations  $\widetilde{\mathbf{Z}_{:,l}}=\big(\widetilde{Z_{-\frac{N_{\tau}}{2},l}},\widetilde{Z_{-\frac{N_{\tau}}{2}+1,l}},\ldots,
 \widetilde{Z_{\frac{N_{\tau}}{2},l}}\big)$ {for
 $l=1,2,\ldots,2\hat{d}$}
  and the
$\mathcal{F} \widetilde{\mathbf{Z}}$ denotes the discrete Fourier
transform acting on each $\widetilde{\mathbf{Z}_{:,l}}$ of $
\widetilde{\mathbf{Z}}$. Then  the fully discrete scheme (FS-F) can
{read}
  \begin{equation*}
\begin{aligned}
&Z_{\mathcal{M},l}^{ni}(\tau)=\sum\limits_{k\in \mathcal{M}}
\widetilde{Z^{ni}_{k,l}}\mathrm{e}^{\mathrm{i} k \tau },\ \
Z_{\mathcal{M},l}^{n+1}(\tau)=\sum\limits_{k\in \mathcal{M}}
\widetilde{Z^{n+1}_{k,l}}\mathrm{e}^{\mathrm{i} k \tau },
\end{aligned}
\end{equation*}
for $ i=1,2,\ldots,s$,
 where we consider the following $s$-stage exponential integrators {(\cite{Ostermann}) applied  to}  \eqref{2scale Fourier}:
\begin{equation}
\begin{array}[c]{ll}%
\widetilde{Z^{ni}}&=e^{c_{i}\hh M}\widetilde{Z^{n}}+\varepsilon  \hh \textstyle\sum\limits_{\rho=1}^{s}\bar{a}_{i\rho}(\hh M)
 \Gamma \big(t _n+c_{\rho}\hh,\widetilde{Z^{n\rho}}\big),\\
\widetilde{Z^{n+1}}&=e^{\hh M}\widetilde{Z^{n}}+\varepsilon  \hh \textstyle\sum\limits_{\rho=1}^{s}\bar{b}_{\rho}(\hh M)
 \Gamma \big(t _n+c_{\rho}\hh,\widetilde{Z^{n\rho}}\big).
\end{array}
 \label{erk dingyi}%
\end{equation}
  Here $i=1,2,\ldots,s$ with $s\geq1$, $\hh $ is the time stepsize,   $  n=0,1,\ldots,T/ \hh -1,$ $M:=\mathrm{i}\Omega$, $c_i$  are constants belonging to $[0,1]$, and $\bar{a}_{i\rho}(\hh M)$,
$\bar{b}_{i}(\hh M)$ are matrix-valued functions of $\hh M$.

The above procedure, however,  is unsuitable in practice because of
the computation of Fourier transform coefficients. In order to
{find} an efficient implementation, we now consider the
discrete Fourier transform coefficients instead of Fourier transform
coefficients. {This gives}   the following
scheme.  For a positive integer
 $N_\tau>1$, let
$\tau_k=\frac{2\pi}{N_\tau}k$ with $k\in \mathcal{M}:=\{-N_\tau/2,-N_\tau/2+1,\ldots,N_\tau/2\}$
and
$Z_{k,l}^{ni}\approx Z_l( t _n+c_i\hh,\tau_k),\ Z_{k,l}^{n}\approx Z_l( t _n,\tau_k)$ for $l=1,2,\ldots,2\hat{d}.$ An
  exponential Fourier spectral discretization (FS-D) is defined as
  \begin{equation} \begin{aligned}
&Z_{j,l}^{ni} =\sum\limits_{k\in \mathcal{M}}
\widehat{Z^{ni}_{k,l}}\mathrm{e}^{\mathrm{i} k \tau_{j}  },\ \
 Z_{j,l}^{n+1} =\sum\limits_{k\in \mathcal{M}}
\widehat{Z^{n+1}_{k,l}}\mathrm{e}^{\mathrm{i} k \tau_{j}  },
\end{aligned} \label{cs ei-ful1}%
\end{equation}
where $j\in \mathcal{M}$ and
\begin{equation}\label{cs ei-ful2}
\begin{array}[c]{ll}%
\widehat{Z^{ni}}=e^{c_{i}\hh M}\widehat{Z^{n}}+\varepsilon  \hh \textstyle\sum\limits_{\rho=1}^{s}\bar{a}_{i\rho}(\hh M)
\Gamma \big(t _n+c_{\rho}\hh,\widehat{Z^{n\rho}}\big),\\
\widehat{Z^{n+1}}=e^{\hh M}\widehat{Z^{n}}+\varepsilon  \hh \textstyle\sum\limits_{\rho=1}^{s}\bar{b}_{\rho}(\hh M)
\Gamma \big( t _n+c_{\rho}\hh,\widehat{Z^{n\rho}}\big).
\end{array}\end{equation}
Here  $0<\hh<1$ is the time stepsize and $s\geq1$ is the stage of the implicit exponential integrator with the coefficients $c_i$, $\bar{a}_{i\rho}(\hh M)$ and
$\bar{b}_{\rho}(\hh M)$.

Based on the above three steps, we obtain the fully discrete scheme, which is described as follows.

\begin{definition}\label{defFDS}
 (\textbf{Fully discrete scheme})
 The fully discrete scheme for the   nonlinear relativistic Klein--Gordon equation   \eqref{klein-gordon}  is defined as follows.

\begin{itemize}
  \item  With the    two-scale technologies stated in Steps 1-2, we get a  highly oscillatory two-scale
  system  \eqref{2scale compact} with the  initial condition \eqref{inv}  over {the  time} interval $[0,T]$.

    \item Consider the Fourier spectral
collocation  (\cite{Shen}) for the approximation of the  operator $-\Delta$ and then a
discrete  two-scale system   \eqref{dis-2scale compact}  over   $[0,T]$ with the initial value  \eqref{dis-inv} is obtained.

  \item   Then
  choose  a positive time stepsize $\hh$, and use the
  exponential Fourier spectral discretization (FS-D)  \eqref{cs ei-ful1}
   to solve  \eqref{dis-2scale compact}.
This produces the numerical approximation $$Z_{l}^{n}:=\sum\limits_{j\in \mathcal{M}}\widehat{Z_{j,l}^{n}}\fe^{\mathrm{i}j n\hh/\eps^2}\approx   Z (n\hh,n\hh/\eps^2)$$   of \eqref{dis-2scale compact}
 for $ l=1,2,\ldots,2\tilde{d}$ and $n=1,2,\ldots,T/ \hh $.

  \item  Finally, the numerical approximation $u^{n}\approx u(\cdot,n\hh  ),\ v^{n}\approx v(\cdot,n\hh  )$ of the original system \eqref{klein-gordon} is given by
{\begin{equation*} \begin{aligned}
u^{n}= &\sqrt{1+\eps^2 A} ^{-1} \Big[\cos\big(  n\hh/\eps^2+ n\hh D_{A}\big)(Z_{l}^{n})_{l=1,2,\ldots,\tilde{d}}\\&
\qquad\qquad\qquad+  \sin\big(  n\hh/\eps^2+ n\hh D_{A}\big) (Z_{\tilde{d}+l}^{n})_{l=1,2,\ldots,\tilde{d}}\Big],\\
 v^{n}=&\eps ^{-2}\Big[ -  \sin\big(  n\hh/\eps^2+ n\hh D_{A}\big)(Z_{l}^{n})_{l=1,2,\ldots,\tilde{d}}\\&\qquad +\cos\big(  n\hh/\eps^2+ n\hh D_{A}\big)(Z_{\tilde{d}+l}^{n})_{l=1,2,\ldots,\tilde{d}}\Big],
 \end{aligned}
\end{equation*}}where $n=1,2,\ldots,T/ \hh .$
\end{itemize}

\end{definition}

%


\subsection{Some practical integrators}
The above procedure fails to be  practical  unless the coefficients $c_i$, $\bar{a}_{i\rho}(\hh M)$ and
$\bar{b}_{i}(\hh M)$ appearing in \eqref{cs ei-ful2} are determined. To this end,  the symmetry and stiff order conditions of \eqref{cs ei-ful2}  are needed.

\begin{proposition} \label{symmetric thm} (\textbf{Symmetric conditions})
The $s$-stage implicit exponential integrator   \eqref{cs ei-ful2} is symmetric
if and only if  for $i,\rho=1,2,\ldots,s$, its coefficients satisfy
\begin{equation}
\begin{array}[c]{ll}%
c_{i}=1-c_{s+1-i},\ \bar{b}_{\rho}(\hh M)=e^{\hh M}\bar{b}_{s+1-\rho}(-\hh M),\\
\bar{a}_{i\rho}(\hh M)=e^{c_{i}\hh M}\bar{b}_{s+1-\rho}(-\hh M)-\bar{a}_{s+1-i,s+1-\rho}(-\hh M).
\end{array}\label{erkdc}%
\end{equation}

\end{proposition}
\begin{proof} Under the conditions \eqref{erkdc},   it is trivial to verify that the method \eqref{cs ei-ful2} remains the same after exchanging $n+1\leftrightarrow n$ and $\hh\leftrightarrow -\hh$. This completes the proof immediately.
\end{proof}

\begin{proposition} \label{stiff thm} (\textbf{Stiff order conditions})
 Define
 \begin{equation*}
\begin{aligned} &  \psi_{\rho}(z)=\varphi_{\rho}(z)-\sum_{i=1}^s\bar{b}_i(z)\frac{c_i^{\rho-1}}{(\rho-1)!},\\&
\psi_{\rho,i}(z)=\varphi_{\rho}(c_iz)c_i^{\rho}-\sum_{\tilde{i}=1}^s\bar{a}_{i\tilde{i}}(z)\frac{c_{\tilde{i}}^{\rho-1}}{(\rho-1)!},\ \ i=1,2,\ldots,s,
 \end{aligned}
\end{equation*}
where the  notations $\varphi_{\rho}$  (\cite{Ostermann}) are defined by $\varphi_{\rho}(z)=\int_0^1
\theta^{\rho-1}\fe^{(1-\theta)z}/(\rho-1)!d\theta$ for $\rho=1,2,\ldots $.
For a fixed number $1\leq r\leq4,$ the order
conditions of Table  \ref{tab1} are assumed to be true up to order $r$ and the condition $\psi_{r}(\hh M)=0$ is weakened to the form $\psi_{r}(0)=0$. Under these assumptions,  the conditions of Assumption \ref{ass},    and the
local assumptions of $\widetilde{Z^{n}}=\widetilde{Z( t _n)}$,
 there exists a constant $\hh_0$ independent of $\eps$ such that for $0<\hh\leq \hh_0$,
  the local error bounds
satisfy the following inequalities
\begin{equation*}
\begin{aligned} & \|\widetilde{Z^{ni}}-\widetilde{Z( t _{n}+c_i\hh)}\|_{H^{\nu-r}}\leq C   \hh^{r}, \\&
  \|\widetilde{Z^{n+1}}-\widetilde{Z( t _{n+1})}\|_{H^{\nu-r}}\leq C \big(  \hh^{r}\norm{\psi_{r}(\hh M)}_{H^{\nu-r+1}}+  \hh^{r+1}\big), \end{aligned} 
\end{equation*}
where $ 0\leq n\leq T/ \hh,$ and $C>0$ is a constant depending on $T$ but is independent of $\eps$ and $\hh$.

\begin{table}[t!]
\renewcommand{\arraystretch}{1.5}
\centering
\begin{tabular}
[c]{|c|c|}\hline Stiff order conditions & Order $r$\\ \hline
$\psi_{1}(\hh M)=0$ & 1\\\hline
$\psi_{2}(\hh M)=0,\ \psi_{1,i}(\hh M)=0$ & 2\\\hline
$\psi_{3}(\hh M)=0,\  \psi_{2,i}(\hh M)=0$ & 3\\\hline
$\psi_{4}(\hh M)=0,\ \psi_{3,i}(\hh M)=0$ & 4\\ \hline
\end{tabular}
\caption{Stiff order conditions. }%
\label{tab1}%
\end{table}
\end{proposition}
\begin{proof}%
The  proof will be given in Section \ref{sec:4}  combined with the analysis of convergence.
\end{proof}

The practical integrators {presented below} will be based
on these symmetric conditions and stiff order conditions.

\textbf{Third-order integrator}.
  We first consider two-stage integrators, i.e., $s=2$.  Solving the order conditions
$\psi_{1}(\hh M)=
\psi_{2}(\hh M)=0
$
leads to $$\bar{b}_{1}(\hh M)= \frac{-c_2 \varphi_1(\hh M) + \varphi_2(\hh M)}{c_1 - c_2},\ \
\bar{b}_{2}(\hh M)= \frac{c_1 \varphi_1(\hh M) - \varphi_2(\hh M)}{c_1 - c_2}.$$
Then using some other order conditions
$$
 \psi_{1,1}(\hh M)= \psi_{1,2}(\hh M)=\psi_{2,2}(\hh M)=0$$
and a symmetric condition
$$\bar{a}_{12}(\hh M)+ \bar{a}_{21}(-\hh M)=\varphi_0(c_1\hh M) \bar{b}_1(-\hh M),$$
we get the results of $\bar{a}_{i\rho}$ as
\begin{equation*}
\begin{aligned}
&\bar{a}_{21}=\frac{c_2^2(-\varphi_{12}+\varphi_{22})}{c_1-c_2},\  \ \bar{a}_{11}=-\bar{a}_{12}+c_1\varphi_{11}, \\&
 \bar{a}_{22}=-\bar{a}_{21}+c_2\varphi_{12}, \
  \bar{a}_{12}=-\bar{a}_{21}(-\hh M)+\varphi_{01} \bar{b}_1(-\hh M),
\end{aligned}
\end{equation*}
where $\varphi_{ij}:=\varphi_i(c_j\hh M)$.
 On noticing $c_1=1-c_2$, it can be {verified} that this class of integrators is at least order two.
 As an example, we choose $c_1=\frac{3-\sqrt{3}}{6}$ which is obtained by requiring a further condition $\psi_{3}(0)=0$ and denote the corresponding method
  as \textbf{S2O3}. For this method, it can be easily checked that $\psi_{2,i}(\hh M)=0$ for $i=1,2$ but $\psi_{3}(\hh M)\neq 0$. Thence it is of order three.

\textbf{Fourth-order integrator}. We now continue with three-stage ($s=3$) integrators and obtain their coefficients by solving
$$\psi_{i}(\hh M)=0\   \textmd{and}  \ \psi_{ \rho,i}(\hh M)=0\   \textmd{for}  \  i,\rho=1,2,3.
$$
The choice of $c_1=1,\   c_2=1/2,\ c_3=0$ and the corresponding results
  \begin{equation*}
\begin{aligned} &\bar{a}_{31}=\bar{a}_{32}=\bar{a}_{33}=0, \  \
\bar{a}_{21}=-\frac{1}{4} \varphi_{22} + \frac{1}{2}\varphi_{32} ,\ \
 \bar{a}_{22}=\varphi_{22} -\varphi_{32}, \\&\bar{a}_{23}=\frac{1}{2}\varphi_{12}-\frac{3}{4}\varphi_{22} +\frac{1}{2}\varphi_3,\ \
 \bar{a}_{11}=\bar{b}_{1}= 4\varphi_3  - \varphi_2 ,\ \ \ \
 \bar{a}_{12}=\bar{b}_{2}=4\varphi_2 -8\varphi_3,\\
 &\bar{a}_{13}=\bar{b}_{3}=\varphi_1-3\varphi_2 +4\varphi_3,\\
\end{aligned}
\end{equation*}
determine this integrator.  It is noted that this method satisfies all the stiff order conditions of order four and symmetric conditions.
 This integrator is referred as \textbf{S3O4}.

%

\section{Main results and numerical tests}\label{sec:3}
In this section, we shall present the main results of this paper. The first one is about convergence and the second is devoted to long time {energy near conservation.}
To support these two results, a numerical experiment  with numerical results is carried out in the second part of this section.
\subsection{Main results}\label{sec:31}
\begin{theorem} \label{UA thm} (\textbf{Convergence}) Under the conditions of Assumption \ref{ass}, Propositions \ref{TS M} and \ref{stiff thm}, for the final numerical solutions  $u^n,\ v^n$  produced
by Definition \ref{defFDS} with S2O3 or S3O4, the global errors are \begin{equation*}
\begin{aligned}
\textmd{S2O3}:\ &\norm{u^n-u(\cdot,n\hh\eps)}_{H^{\nu-1}}  \leq {C_1} (  \hh^3+\delta_{\mathcal{F}}),\\
&\norm{v^n-v(\cdot,n\hh\eps)}_{H^{\nu-2}} \leq {C_2} ( \hh^3/\eps^2+\delta_{\mathcal{F}}),\\
\textmd{S3O4}:\ &\norm{u^n-u(\cdot,n\hh\eps)}_{H^{\nu-3}}  \leq {C_3} (  \hh^4+\delta_{\mathcal{F}}),\\ &\norm{v^n-v(\cdot,n\hh\eps)}_{H^{\nu-4}} \leq {C_4} (  \hh^4/\eps^2+\delta_{\mathcal{F}}),\\
\end{aligned}
\end{equation*}
where $0\leq n\leq T/\hh$,  {the constants $C_1$-$C_4$ depend on $T$,  $\norm{\psi_1(x)}_{H^{\nu+1}}$ and $\norm{\psi_2(x)}_{H^{\nu}}$,} but are independent of $n, \hh, \eps $. Here $\delta_{\mathcal{F}}$  denotes {the  error in $x$ and $\tau$ brought by the Fourier pseudo-spectral method.} {We point out that $\nu\geq 4+d$ is required for S3O4 and  it can be weakened to $\nu\geq 2+d$ for S2O3.}
\end{theorem}
\begin{remark}
It is noted that these two methods have a very nice  accuracy  which is $\mathcal{O}_{H^{\nu-1}}( \hh^3)$ for S2O3 and $\mathcal{O}_{H^{\nu-3}}( \hh^4)$  for S3O4 in the approximation of  $u$. 
This high accuracy is very significant for the numerical methods applied to highly oscillatory systems where $\eps$ is a very small value.
\end{remark}

The next theorem requires a non-resonance condition and to describe it, we introduce the   notations (\cite{WZ21})
\begin{equation*}\begin{array}[c]{ll}
\mathbf{k}:=&\big(k_{-\frac{N_{\tau}}{2},1},k_{-\frac{N_{\tau}}{2}+1,1},\ldots,
k_{\frac{N_{\tau}}{2},1},k_{-\frac{N_{\tau}}{2},2},k_{-\frac{N_{\tau}}{2}+1,2},\ldots,
 k_{\frac{N_{\tau}}{2},2},\\&\qquad\ldots,k_{-\frac{N_{\tau}}{2},2\hat{d}},\ldots,
k_{\frac{N_{\tau}}{2},2\hat{d}}\big),\\
 \boldsymbol{\omega}:=&\big(\omega_{-\frac{N_{\tau}}{2},1},\omega_{-\frac{N_{\tau}}{2}+1,1},\ldots,
\omega_{\frac{N_{\tau}}{2},1},\omega_{-\frac{N_{\tau}}{2},2},\omega_{-\frac{N_{\tau}}{2}+1,2},\ldots,
\omega_{\frac{N_{\tau}}{2},2},\\&\qquad\ldots,\omega_{-\frac{N_{\tau}}{2},2\hat{d}},\ldots,
\omega_{\frac{N_{\tau}}{2},2\hat{d}}\big),\\
|\mathbf{k}|:=&\sum_{l=1}^{2\hat{d}}\sum_{j=-\frac{N_{\tau}}{2}}^{\frac{N_{\tau}}{2}}
|k_{j,l}|,\ \ \ \mathbf{k}\cdot \boldsymbol{\omega}:=\sum_{l=1}^{2\hat{d}}\sum_{j=-\frac{N_{\tau}}{2}}^{\frac{N_{\tau}}{2}}
k_{j,l}\omega_{j,l},
\end{array}
\end{equation*}
where
$\big(\omega_{-\frac{N_{\tau}}{2},l},\omega_{-\frac{N_{\tau}}{2}+1,l},\ldots,
\omega_{\frac{N_{\tau}}{2},l}\big):=\frac{1}{\eps^2}\big(\frac{N_{\tau}}{2},
  \frac{N_{\tau}}{2}-1,\ldots,-\frac{N_{\tau}}{2}\big)$ for any $1\leq l\leq 2\hat{d}$.
Denote the resonance module by $\mathcal{M}:= \{\mathbf{k}\in \mathcal{Q}:\ \mathbf{k}\cdot
\boldsymbol{\omega}=0\},$  where $\mathcal{Q}:= \{\mathbf{k}\in \mathbb{Z}^{D}:  \textmd{there exists an}\ l\in \{1,\ldots,2\hat{d}\}\
 \textmd{such that}\
|\mathbf{k}_{:,l}|=|\mathbf{k}|\}$.
Let  $\langle j\rangle_l$ be the unit coordinate vector $(0, \ldots , 0, 1, 0,
\ldots,0)^{\intercal}\in \mathbb{R}^{D}$  with the only entry $1$ at the $(j,l)$-th
position. Further let $\mathcal{K}$ be  a set of representatives of the
equivalence classes in $\mathcal{Q}/\mathcal{M}$. The set $\mathcal{K}$ is determined by two requirements.
The first is that if $\mathbf{k}\in\mathcal{K}$, we have $-\mathbf{k}\in\mathcal{K}.$ The other  is to minimize  the sum $|\mathbf{k}|$ in the equivalence class $[\mathbf{k}] = \mathbf{k} +\mathcal{M}$  for each $\mathbf{k}\in\mathcal{K}$.
Meanwhile,  for those elements  having the same minimal sum $|\mathbf{k}|$,  all of them are kept in $\mathcal{K}$.
Denote $\mathcal{N}_N =\{\mathbf{k}\in\mathcal{K}:\ |\mathbf{k}|\leq N\}$ and $
\mathcal{N}_N^*=\mathcal{N}_N \bigcup
\{\langle 0\rangle_l\}_{l=1,2,\ldots,2\hat{d}}$
for a positive
integer $N$.
\begin{theorem} \label{Long-time thm} (\textbf{Long time energy {near conservation}}) Denote the initial value appeared in
Definition \ref{defFDS} by $\delta_0:=\norm{\widehat{Z^{0}}}_{H^{\nu}} $ which is a value between 0 and 1\footnote{If not,  it can   be achieved by rescaling the solution of the considered system.}. The non-linearity $f(u)$ is assumed to be   smooth and satisfy $f(0)=f'(0)=0.$ For  the time stepsize $\hh$, we
assume a lower bound   $\hh/\sqrt{\eps} \geq c_0 > 0$ and $\hh\leq \delta_0$\footnote{It is noted that this restriction is an artificial  condition for rigorous proof and is hoped to be weakened in future.}.
It is further
required that the  numerical non-resonance condition $
|\sin\big(\frac{1}{2}\hh \omega_{j,1}\big)| \geq c_1 \sqrt{\hh}
$ holds
for a constant $c_1>0$ and  $j=-\frac{N_{\tau}}{2},-\frac{N_{\tau}}{2}+1,\ldots,\frac{N_{\tau}}{2}$.
Then  the long time energy
conservation of $u^n,\ v^n$  produced
by Definition \ref{defFDS} with S2O3 or S3O4 is  estimated by
\begin{equation}\label{er}
\begin{aligned}
& {\frac{\eps^2}{\delta_0^2} \abs{H(u^n,v^n)- H(u^0,v^0)}\leq C\eps^{3}\delta_0^2+C\delta_{\mathcal{F}},\ \ 0\leq n\hh\leq \frac{\eps \delta_0^{-N+3}}{\hh^2}},\\
\end{aligned}
\end{equation}
{where the constant $C$  depends on $N, N_{\tau},c_0, c_1$
but is independent of $n,\ h,\ \eps $, and $\delta_{\mathcal{F}}$ denotes   the error brought by the Fourier pseudospectral method.    Since  $N$ can be arbitrarily large and   $ \eps /h^2<1/c_0^2$,  the near conservation law holds  for a long time.}
\end{theorem}

\begin{remark}
The above statement \eqref{er} seems a little surprised since these two methods (with different order) have the same result of energy {near conservation}. It is noted here that this fact comes from  the  same  boundedness of the coefficient functions appeared in the modulated Fourier expansion.  Moreover, the numerical results of the  test given in Section \ref{se;NT}   demonstrate and support this behaviour  (see Figures \ref{p7}-\ref{p8}). \end{remark}

\begin{remark}
Although the smallness of the parameter $\delta_0$ is technically
required,   we should note here that $\delta_0$  is totally
independent of $\eps$. This means that  the initial value of
\eqref{klein-gordon} is large, which has not been considered
yet in the long term analysis of any methods.  In all the previous
work on this topic,  small initial data is required (see, e.g.
\cite{Cohen0,Cohen2,Lubich2006,Lubich2,hairer2000,lubich19,WW,WZ21}).
Moreover, we only need the  lower bound on the time stepsize $\hh \geq
c_0 \sqrt{\eps}$, which means that large {stepsize} can
be used to keep the long time energy conservation. Compared with the analysis of \cite{WZ21}, a looser  numerical non-resonance requirement is posed in this theorem and this is  because the  methods derived in this paper avoid $\sin\big(\frac{1}{2}\hh (\omega_{j,1}- (\mathbf{k}\cdot \boldsymbol{\omega}))\big)$ with $\mathbf{k}\in \mathcal{N}_N$ but $\neq\langle j \rangle_l$ in the denominator of the ansatz \eqref{ansatz} of the modulated
Fourier functions.
\end{remark}

\begin{remark}
It is noted that we only focus on the energy conservation in this paper. For the  nonlinear relativistic Klein--Gordon  equation, it has other (almost) invariants such as momentum and  harmonic actions. With the same arguments presented in this paper, the proposed methods can be proved to have long time {near conservations} in these invariants. This paper will not go further on this aspect for  simplicity.
\end{remark}

\subsection{Numerical test}\label{se;NT}
\subsubsection{1D test} 
As a numerical example, we consider
the
 nonlinear relativistic Klein--Gordon equation \eqref{klein-gordon} with $\lambda=-1,d=1,\Omega_x=(-\pi,\pi)$
 and the initial values $$\psi_1(x)=\frac{3\sin(x)}{\exp(x^2/2)+\exp(-x^2/2)},\  \  \psi_2(x)=\frac{2\exp(-x^2)}{\sqrt{\pi}}.$$
  By the {Fourier} spectral collocation method, we consider the
second-order Fourier differentiation matrix
$(a_{kj})_{\tilde{M}\times \tilde{M}}$ whose   entries  are given by
$$a_{kj}=\left\{
             \begin{array}{ll}
             \frac{(-1)^{k+j}}{2}\sin^{-2}\Big(\frac{(k-j)\pi}{\tilde{M}}\Big),&\quad k\neq j,\\
             \\
             \frac{M^2}{12}+\frac{1}{6},&\quad k=j,
             \end{array}
       \right.$$
with $\tilde{M}=\frac{2\pi}{N_x}$.
In this test,
we consider $N_x=32$ and $N_\tau=64$.  For implicit methods, we use standard fixed point iteration as nonlinear solver in the practical computations. We
set $10^{-12}$ as the error tolerance and $200$ as the maximum number of each iteration.
   \begin{figure}[t!]
\centering
\includegraphics[width=5.2cm,height=5.1cm]{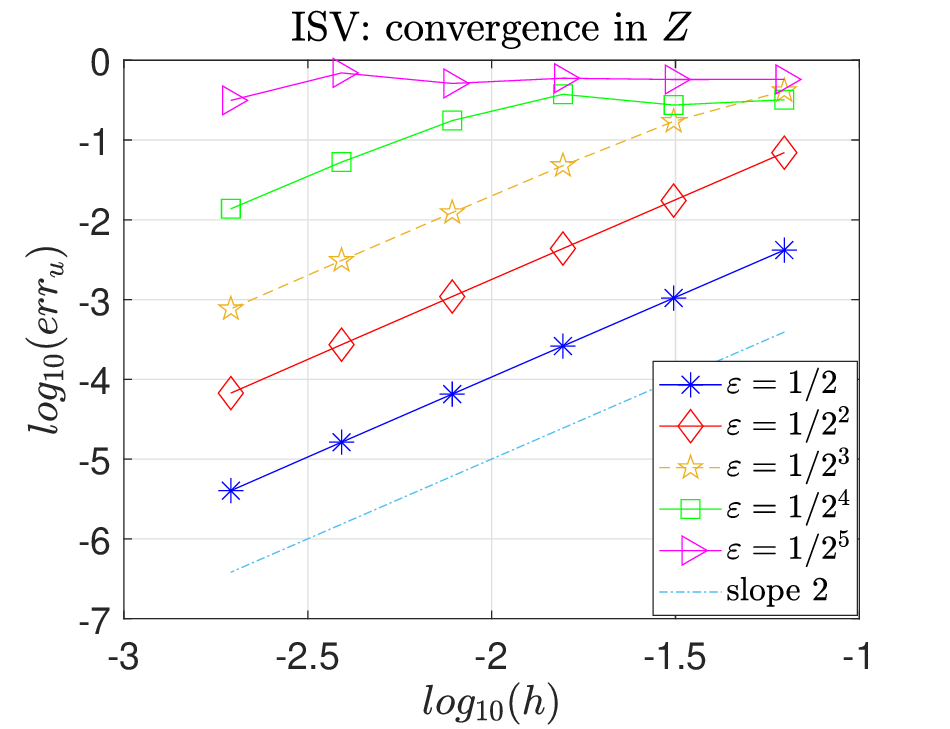}
\includegraphics[width=5.2cm,height=5.1cm]{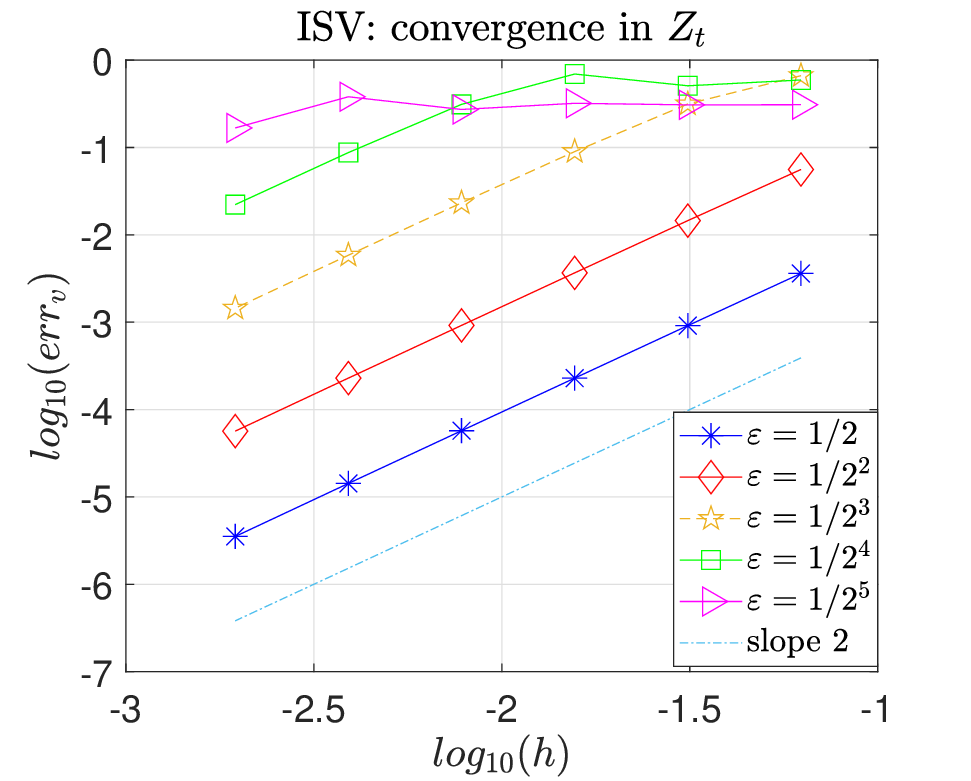}
\caption{ISV: the log-log plot of the temporal errors $err_u=\frac{\norm{u^n-u(\cdot, t _n)}_{H^{1}}}{\norm{u(\cdot, t _n)}_{H^{1}}}$ and $err_v=\frac{\norm{v^n-v( \cdot,t _n)}_{H^{0}}}{\norm{v( \cdot,t _n)}_{H^{0}}}$  at $ t _n=1$ under different $\hh$, where $\eps=1/2^k$ with $k=1,2,\ldots,5$.} \label{p1}
\end{figure}

   \begin{figure}[t!]
\centering
\includegraphics[width=5.2cm,height=5.1cm]{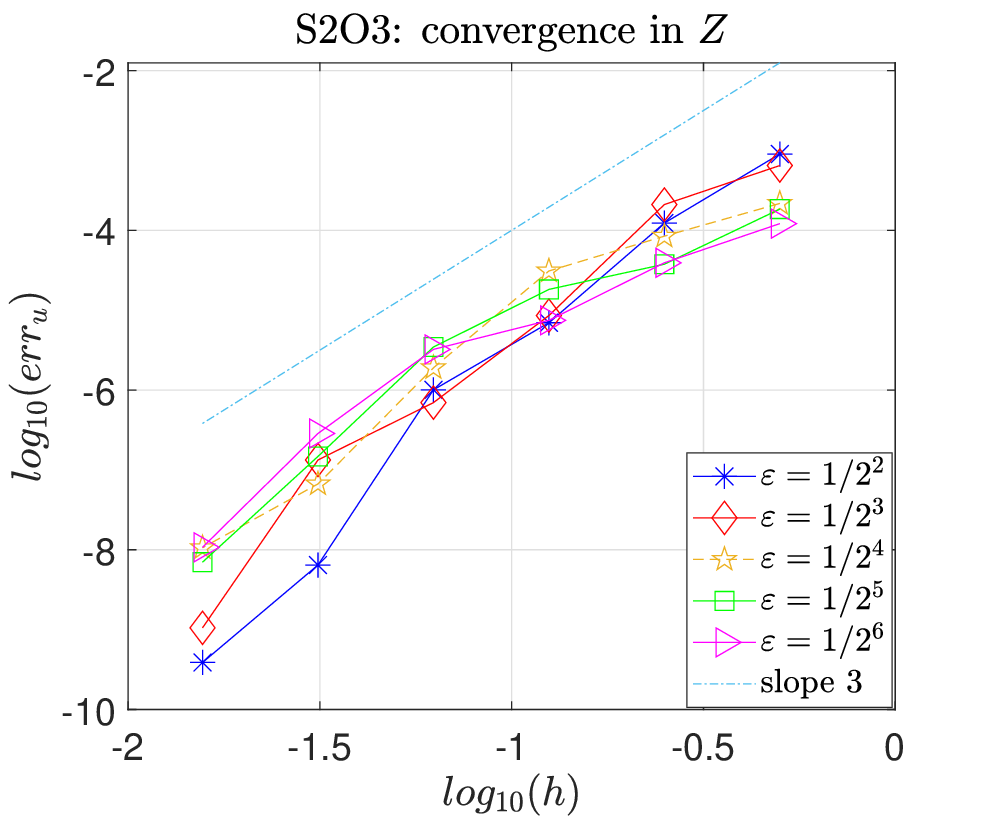}
\includegraphics[width=5.2cm,height=5.1cm]{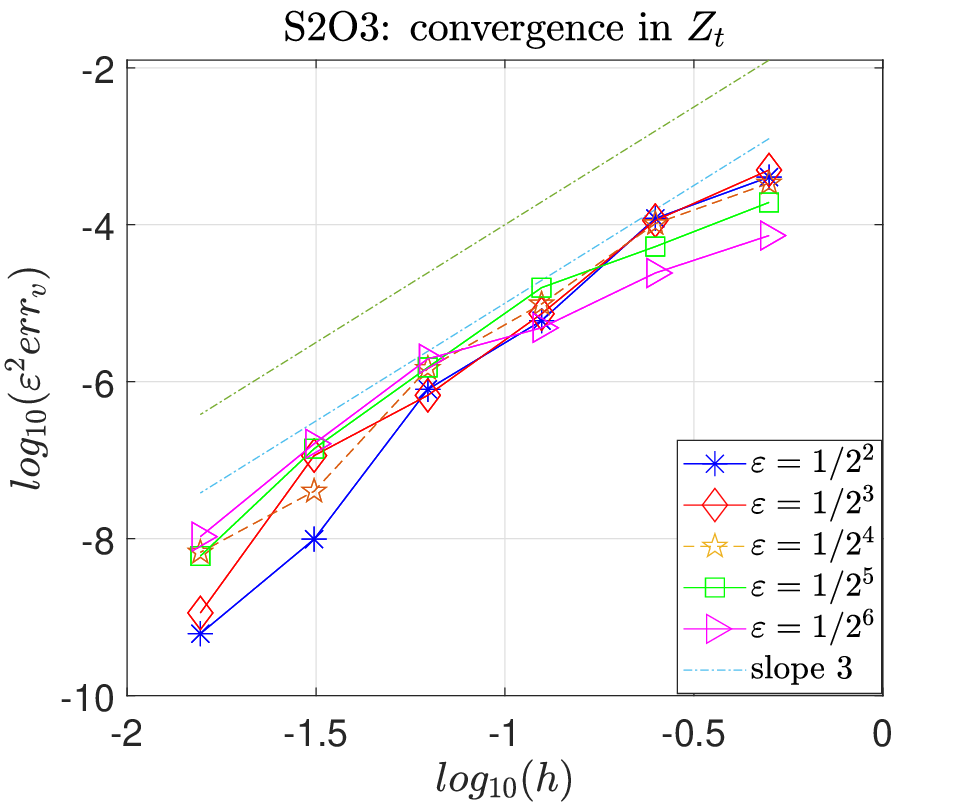}
\caption{S2O3: the log-log plot of the temporal errors $err_u=\frac{\norm{u^n-u(\cdot, t _n)}_{H^{1}}}{\norm{u(\cdot, t _n)}_{H^{1}}}$ and $err_v=\frac{\norm{v^n-v( \cdot,t _n)}_{H^{0}}}{\norm{v( \cdot,t _n)}_{H^{0}}}$  at $ t _n=1$ under different $\hh$, where $\eps=1/2^k$ with $k=1,2,\ldots,5$.} \label{p2}
\end{figure}

   \begin{figure}[t!]
\centering
\includegraphics[width=5.2cm,height=5.1cm]{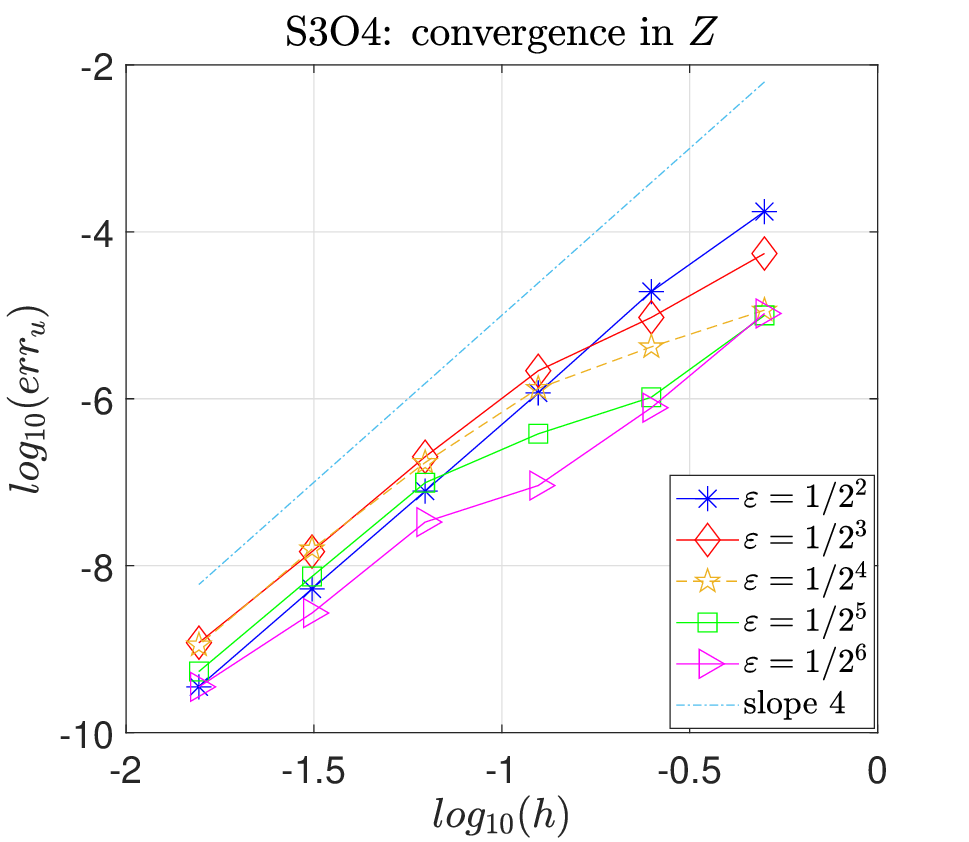}
\includegraphics[width=5.2cm,height=5.1cm]{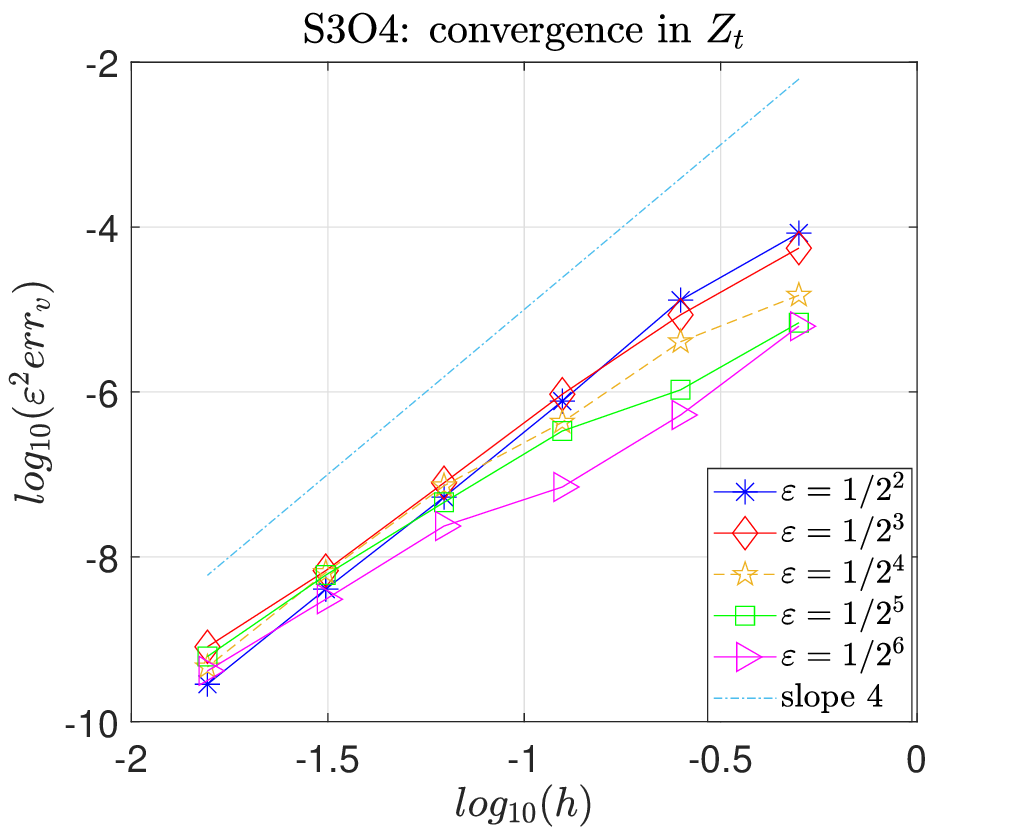}
\caption{S3O4: the log-log plot of the temporal errors $err_u=\frac{\norm{u^n-u(\cdot, t _n)}_{H^{1}}}{\norm{u(\cdot, t _n)}_{H^{1}}}$ and $err_v=\frac{\norm{v^n-v( \cdot,t _n)}_{H^{0}}}{\norm{v( \cdot,t _n)}_{H^{0}}}$  at $ t _n=1$ under different $\hh$, where $\eps=1/2^k$ with $k=1,2,\ldots,5$.} \label{p3}
\end{figure}

\textbf{Accuracy.}  For comparison, we choose  the second-order improved St\"{o}rmer-Verlet method (\textbf{ISV})  given in \cite {hairer2000}. This method is directly used without taking the process given in Section \ref{sec:2}.
For the methods presented in this paper, {they are used to solve the two-scale system  \eqref{2scale compact} over $[0,T]$.}
  Firstly  the accuracy of all the methods is shown by displaying the global errors
 $$err_u=\frac{\norm{u^n-u(\cdot, t _n)}_{H^{1}}}{\norm{u(\cdot, t _n)}_{H^{1}}},\ err_v=\frac{\norm{v^n-v( \cdot,t _n)}_{H^{0}}}{\norm{v( \cdot,t _n)}_{H^{0}}}$$   at $T=1$ in Figures \ref{p1}-\ref{p6}.
We use the result given by the S3O4 with a small time stepsize as the reference solution. In the light of these results, we have the following observations.

a) The  improved St\"{o}rmer-Verlet method ISV shows non-uniform accuracy.  When $\eps$ becomes small, the accuracy becomes badly (see Figure \ref{p1}).

b) The two integrators given in this paper have   uniform   accuracy,  and when $h$ decreases, the accuracy is improved.  S2O3 shows third order and S3O4 performs   fourth order (see Figures \ref{p2}-\ref{p3}).

c) Figures \ref{p4}-\ref{p6}  show the dependence of  global errors in $\eps$. It can be seen that for a fixed time stepsize,
ISV behaves very badly for small $\eps$ (see Figure \ref{p4}), and
S2O3 and  S3O4 have uniform  convergence in $u$ and $\eps^2 v$  (see Figures \ref{p5} and \ref{p6}, respectively).
These observations agree with the  theoretical results given in Theorem \ref{UA thm}.

   \begin{figure}[t!]
\centering
\includegraphics[width=5.2cm,height=5.1cm]{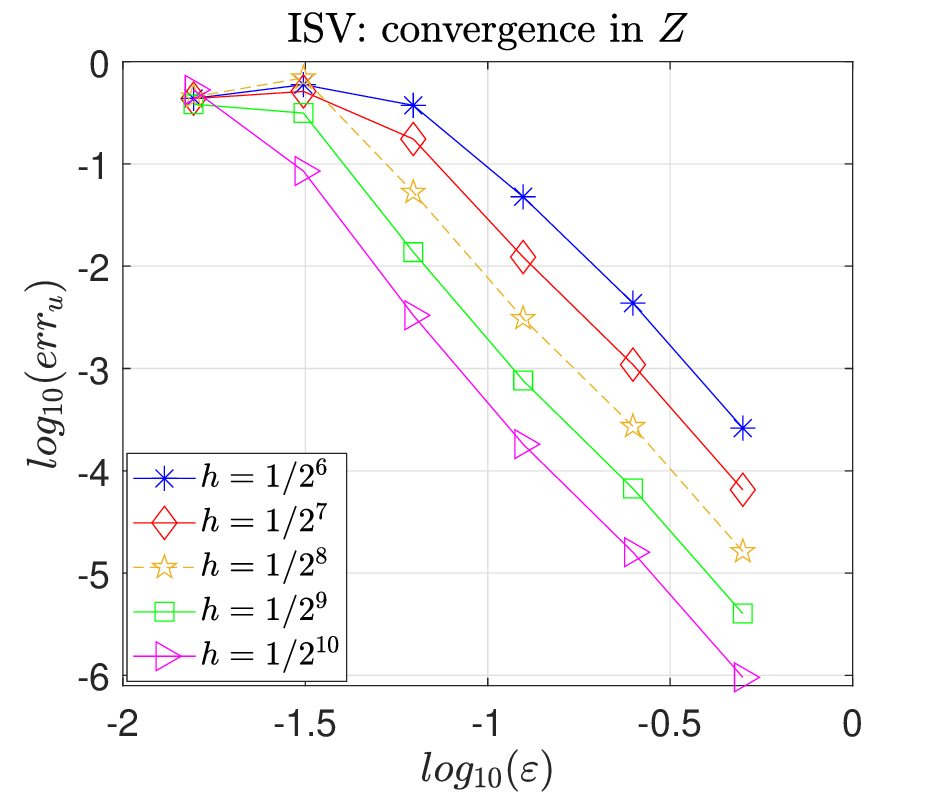}
\includegraphics[width=5.2cm,height=5.1cm]{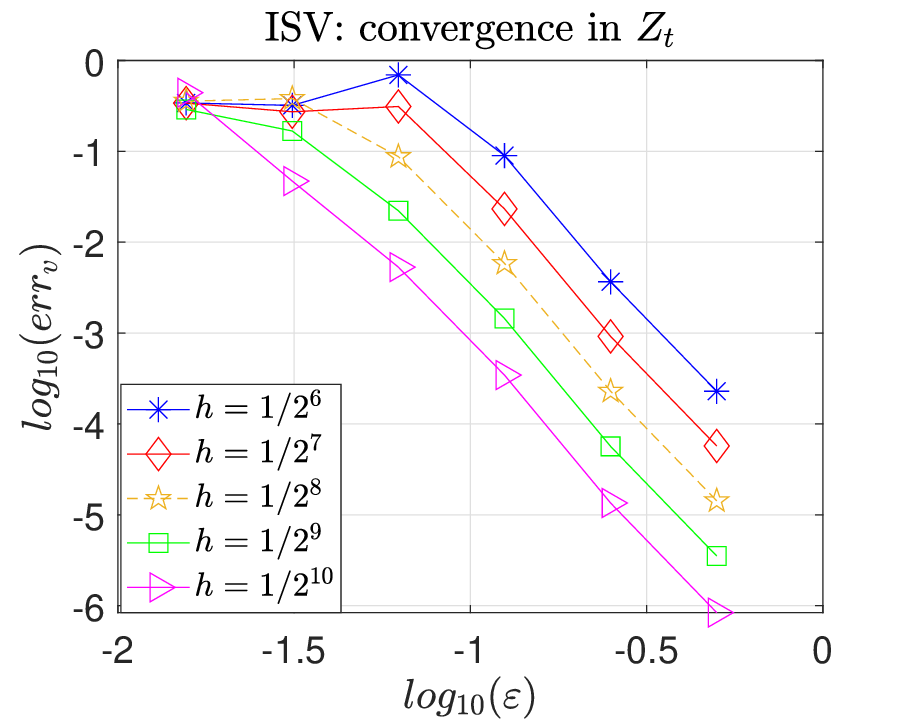}
\caption{ISV: the log-log plot of the temporal errors $err_u=\frac{\norm{u^n-u(\cdot, t _n)}_{H^{1}}}{\norm{u(\cdot, t _n)}_{H^{1}}}$ and $err_v=\frac{\norm{v^n-v( \cdot,t _n)}_{H^{0}}}{\norm{v( \cdot,t _n)}_{H^{0}}}$  at $ t _n=1$ under different $\eps$, where $\hh=1/2^k$ with $k=6,7,\ldots,10$.} \label{p4}
\end{figure}

   \begin{figure}[t!]
\centering
\includegraphics[width=5.2cm,height=5.1cm]{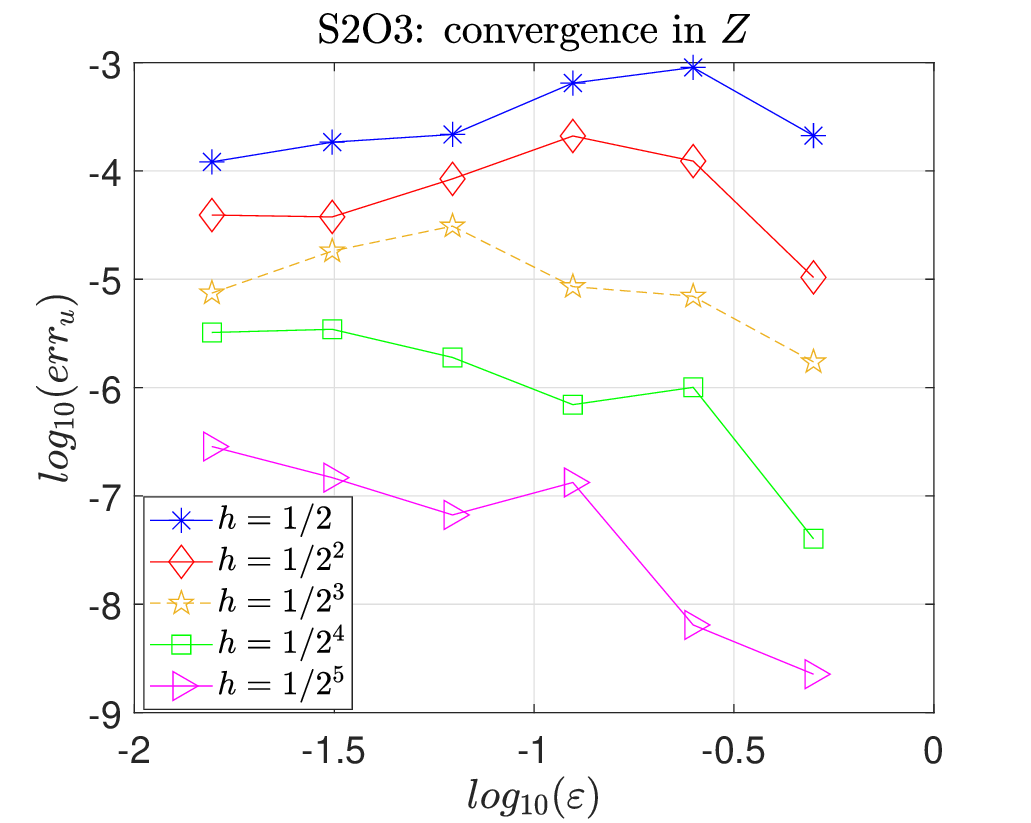}
\includegraphics[width=5.2cm,height=5.1cm]{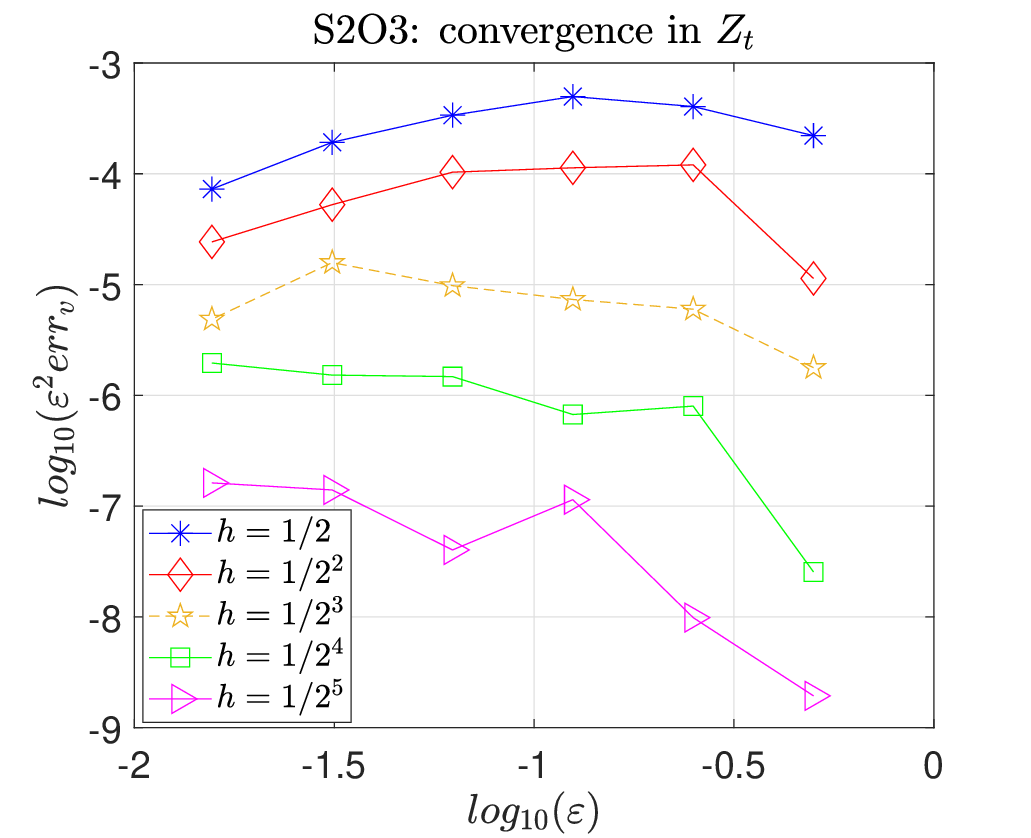}
\caption{S2O3: the log-log plot of the temporal errors $err_u=\frac{\norm{u^n-u(\cdot, t _n)}_{H^{1}}}{\norm{u(\cdot, t _n)}_{H^{1}}}$ and $err_v=\frac{\norm{v^n-v( \cdot,t _n)}_{H^{0}}}{\norm{v( \cdot,t _n)}_{H^{0}}}$  at $ t _n=1$ under different $\eps$, where $\hh=1/2^k$ with $k=6,7,\ldots,10$.} \label{p5}
\end{figure}

   \begin{figure}[t!]
\centering
\includegraphics[width=5.2cm,height=5.1cm]{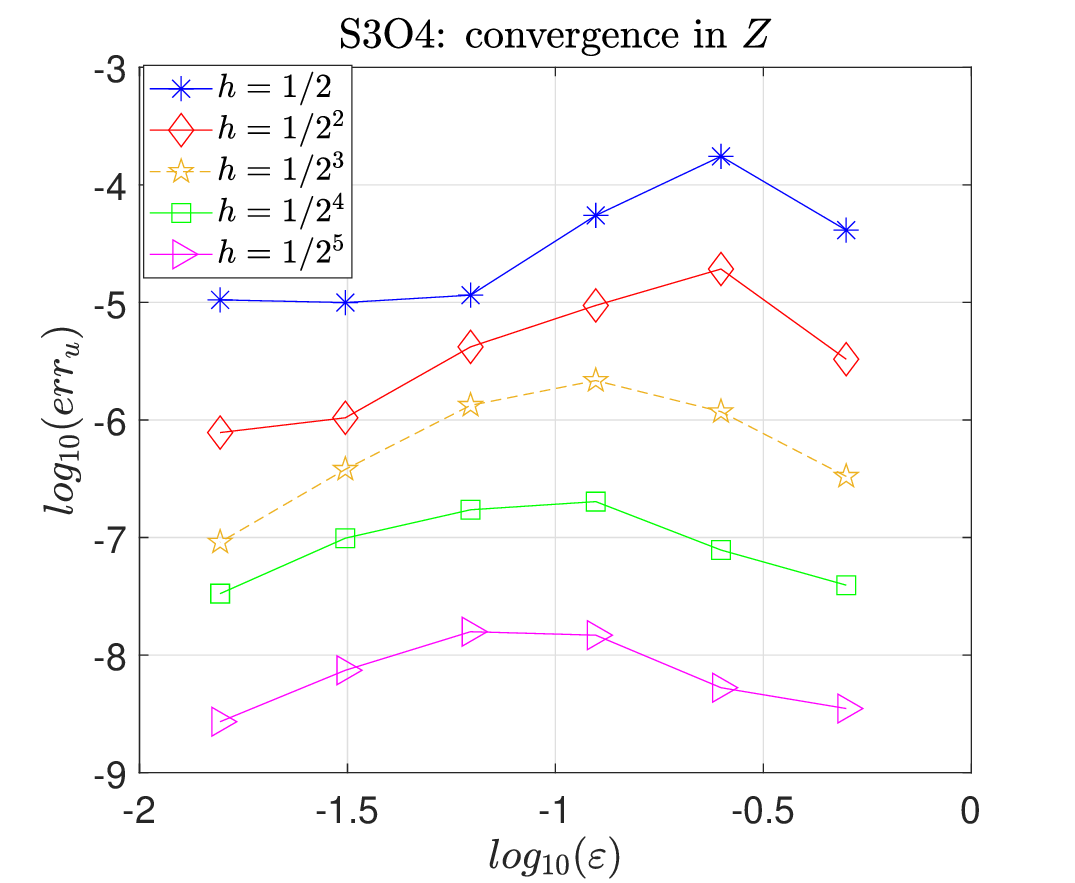}
\includegraphics[width=5.2cm,height=5.1cm]{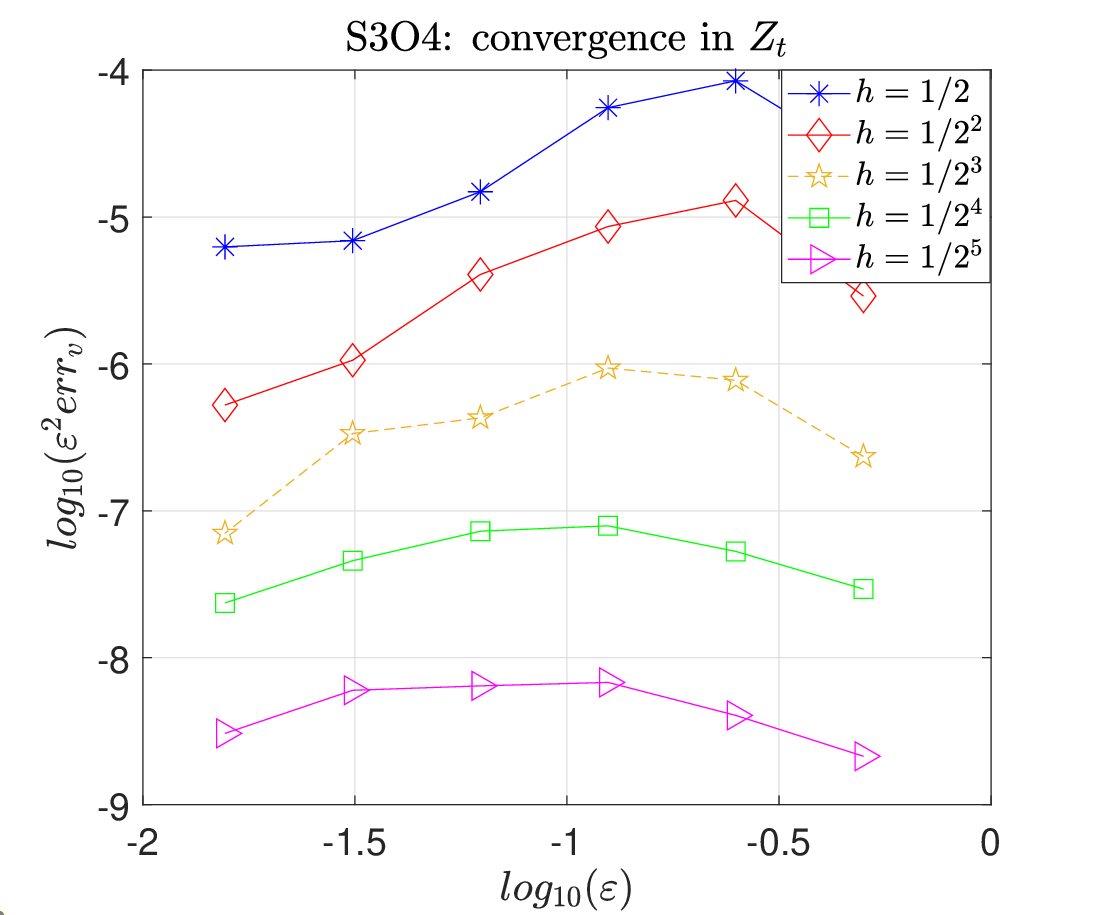}
\caption{S3O4: the log-log plot of the temporal errors $err_u=\frac{\norm{u^n-u(\cdot, t _n)}_{H^{1}}}{\norm{u(\cdot, t _n)}_{H^{1}}}$ and $err_v=\frac{\norm{v^n-v( \cdot,t _n)}_{H^{0}}}{\norm{v( \cdot,t _n)}_{H^{0}}}$  at $ t _n=1$ under different $\eps$, where $\hh=1/2^k$ with $k=6,7,\ldots,10$.} \label{p6}
\end{figure}

 \textbf{Energy conservation.}  Then  we display the long time energy conservation of our methods by presenting the energy error
$err_H=\frac{|H(Z^n,Z_t^n)-H(Z^0,Z_t^0)|}{|H(Z^0,Z_t^0)|}.$   For comparison,  we replace S2O3 with the following one-stage  exponential integrator:{
\begin{equation*} \begin{split}
&\widehat{Z^{n1}}=e^{ \hh M/2}\widehat{Z^{n}}+ h/2\Gamma \big(t_n+h/2,\widehat{Z^{n1}}\big),\\
&\widehat{Z^{n+1}}=e^{\hh M}\widehat{Z^{n}}+  h\varphi_{1}(\hh M)\Gamma \big(t_n+h/2,\widehat{Z^{n1}}\big).
\end{split}\end{equation*} }This method  is non-symmetric and we shall denote it by \textbf{NSM}.
The energy  errors
 are shown in Figures \ref{p7}-\ref{p9} for  different $\eps$ and large   $\hh$.
According to these numerical results, the following observations are made.

a) The energy $H$ is nearly preserved numerically by   our
integrators   over long times. S2O3 and S3O4  have an nice and similar long time conservation, 
and with large time {stepsize} $\hh$, the numerical error in
the energy can be improved when $\eps$ decreases (see Figures \ref{p7}-\ref{p8}).
There observations agree with the results given in Theorem \ref{Long-time thm}.

b) In contrast, NSM shows substantial drift in the energy quantity and thus it  does not have long-term performance in the energy conservation (see Figure  \ref{p9}). The reason is that it is not a symmetric method.
This observation demonstrates that  symmetric conditions   play an important role in the numerical behaviour of energy conservation.

   \begin{figure}[t!]
\centering
\includegraphics[width=5.5cm,height=4.8cm]{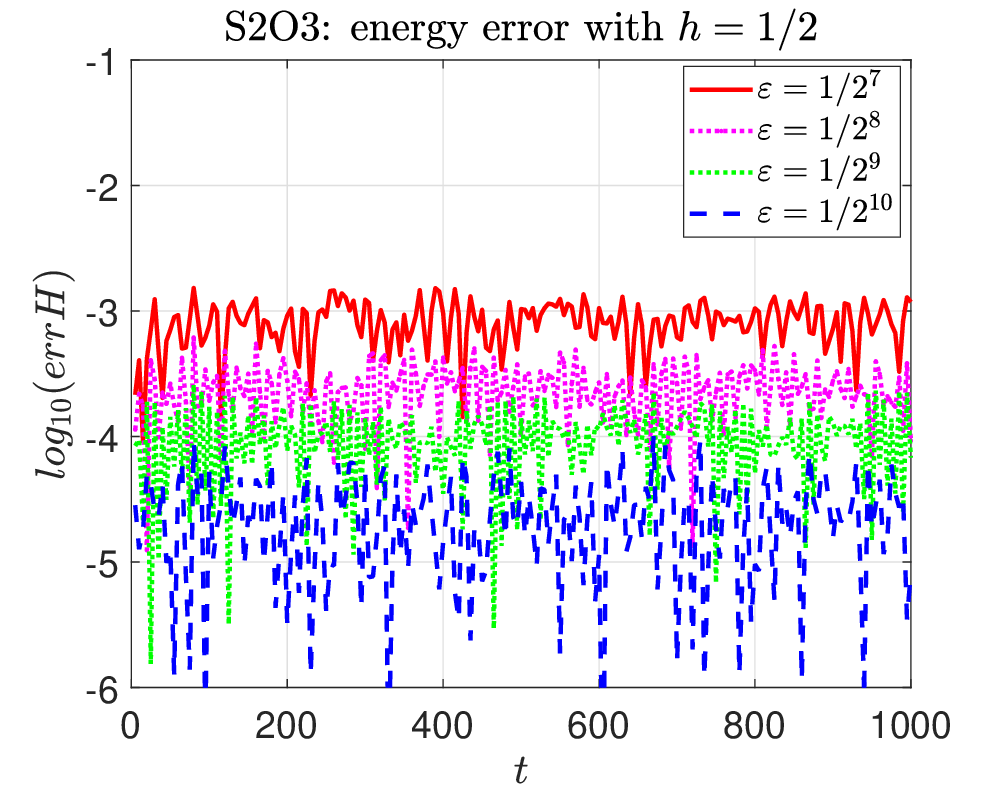}
\includegraphics[width=5.5cm,height=4.8cm]{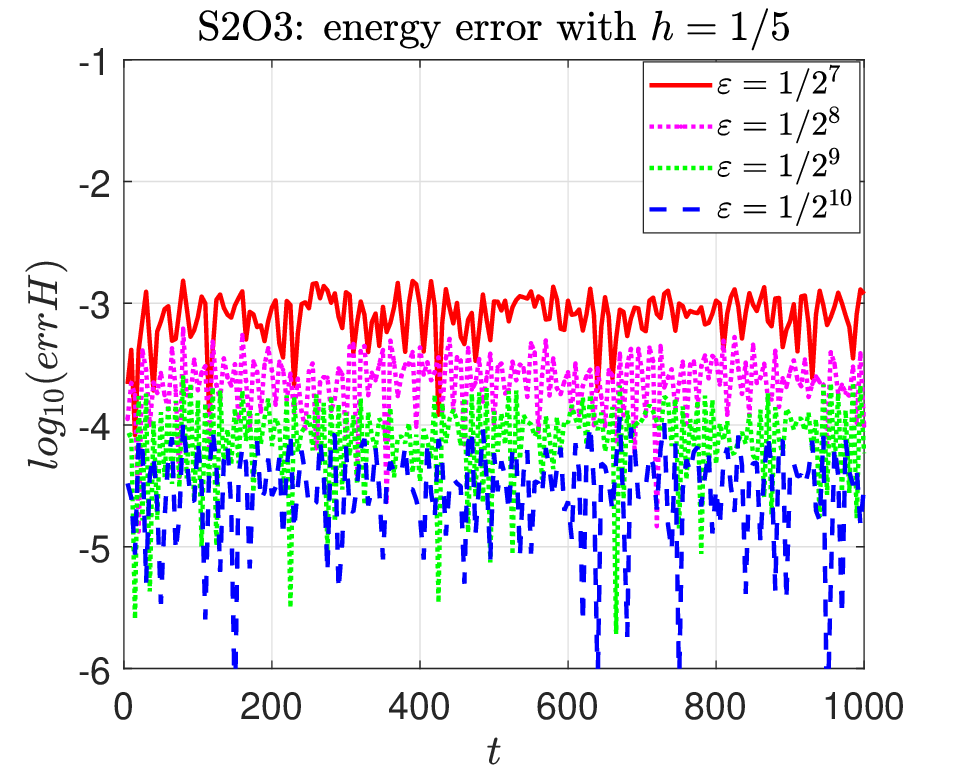}
\caption{S2O3: the plot of the energy error $err_H=\frac{|H(u^n,v^n)-H(u^0,v^0)|}{|H(u^0,v^0)|}$    against $t$.} \label{p7}
\end{figure}

   \begin{figure}[t!]
\centering
\includegraphics[width=5.5cm,height=4.8cm]{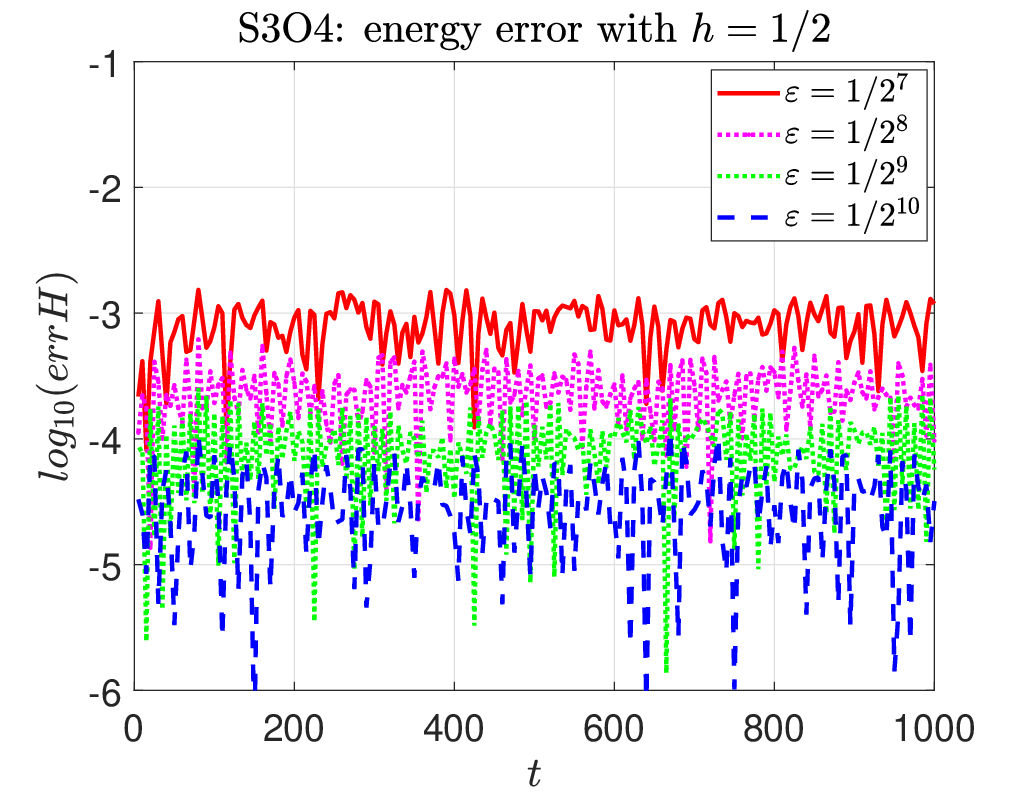}
\includegraphics[width=5.5cm,height=4.8cm]{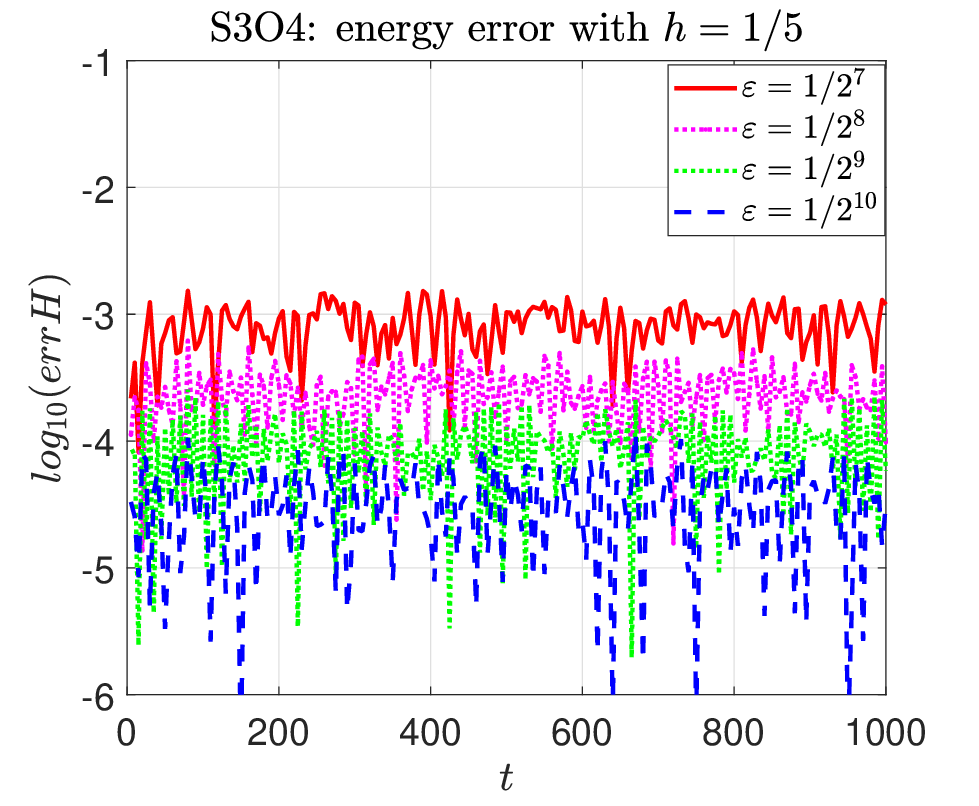}
\caption{S3O4: the plot of the energy error $err_H=\frac{|H(u^n,v^n)-H(u^0,v^0)|}{|H(u^0,v^0)|}$     against $t$.} \label{p8}
\end{figure}

   \begin{figure}[t!]
\centering
\includegraphics[width=5.5cm,height=4.8cm]{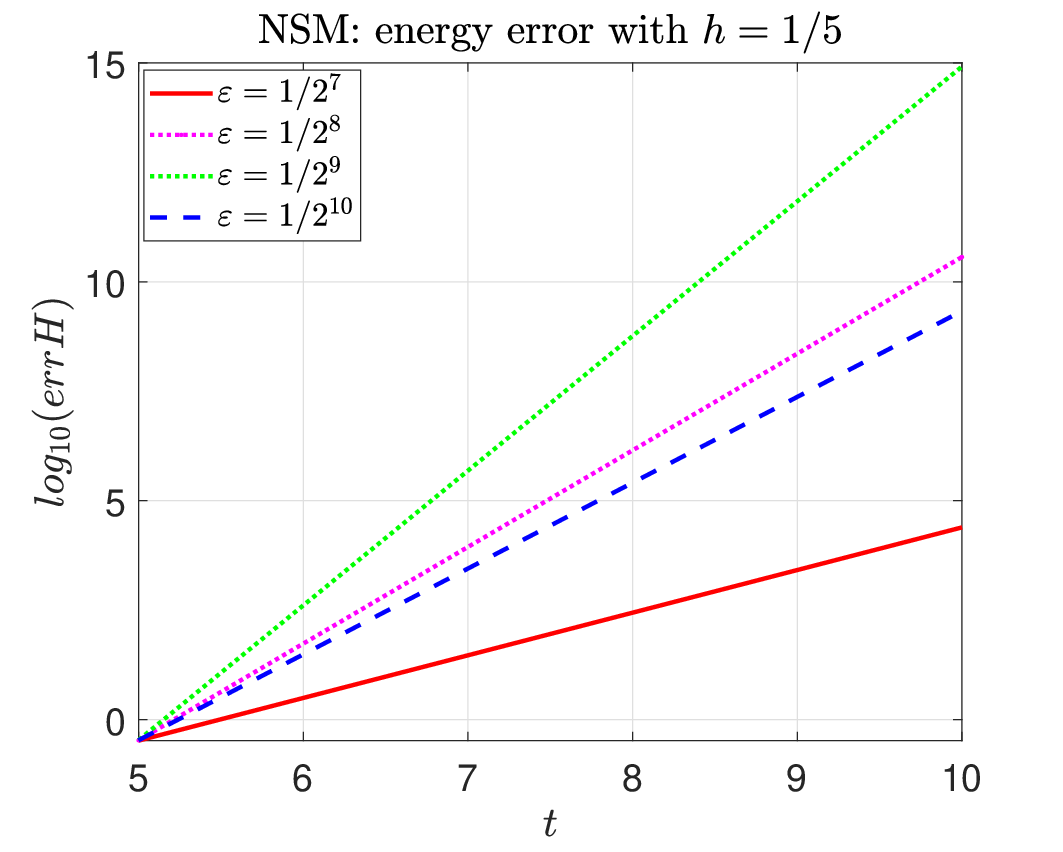}
\includegraphics[width=5.5cm,height=4.8cm]{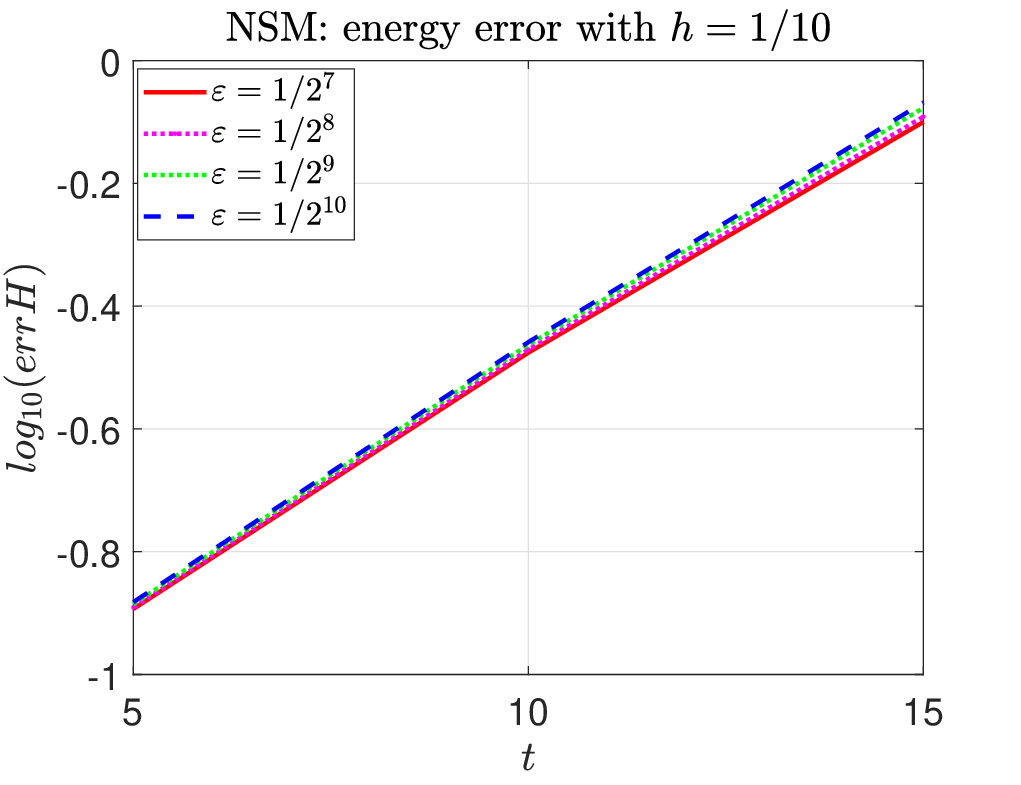}
\caption{NSM: the plot of the energy error $err_H=\frac{|H(u^n,v^n)-H(u^0,v^0)|}{|H(u^0,v^0)|}$     against $t$.} \label{p9}
\end{figure}


\textbf{Efficiency.} Compared with ISV, the scheme of our integrators given in this paper is more complicate. For example, the two-scale method   enlarges the dimension of the original system and this usually adds some computation cost. Fortunately,  Fast Fourier Transform (FFT) techniques can be used in the integrators and we hope that the efficiency of our integrators is still acceptable even compared with methods without using the two-scale  technology. To show this point, we solve this problem on the time interval $[0,T]$. The efficiency\footnote{This test is conducted in a sequential program in MATLAB R2020b
 on a laptop
   ThinkPad X1 nano (CPU:  i7-1160G7 @ 1.20GHz   2.11 GHz, Memory: 16 GB, Os: Microsoft Windows 10 with 64bit). }
 of each integrator  (measured by
the log-log plot of the temporal error  at $t=10$ against CPU time) is displayed in Figure \ref{p10}.
Clearly, our integrators have very competitive efficiency.

   \begin{figure}[t!]
\centering
\includegraphics[width=5.5cm,height=4.8cm]{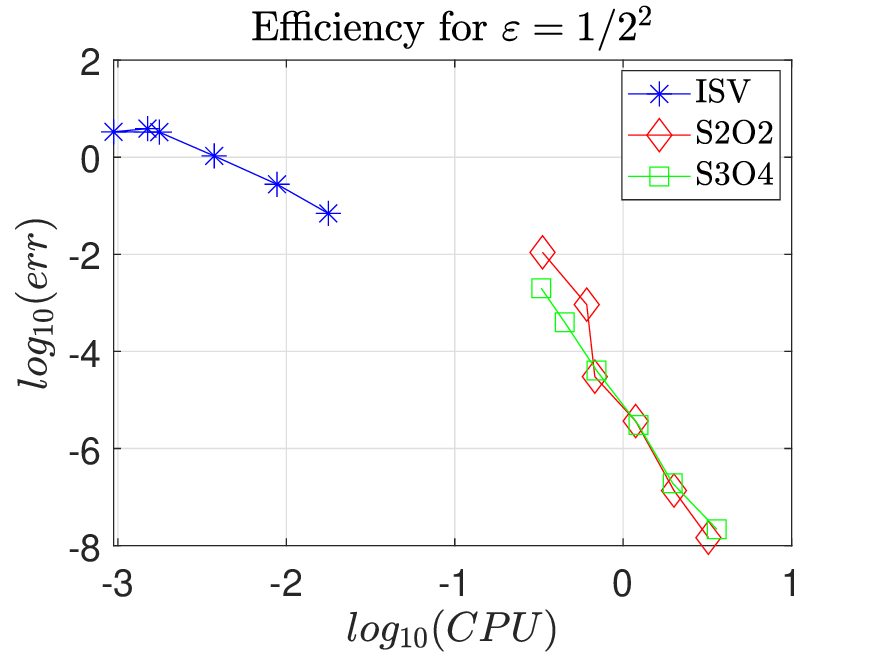}
\includegraphics[width=5.5cm,height=4.8cm]{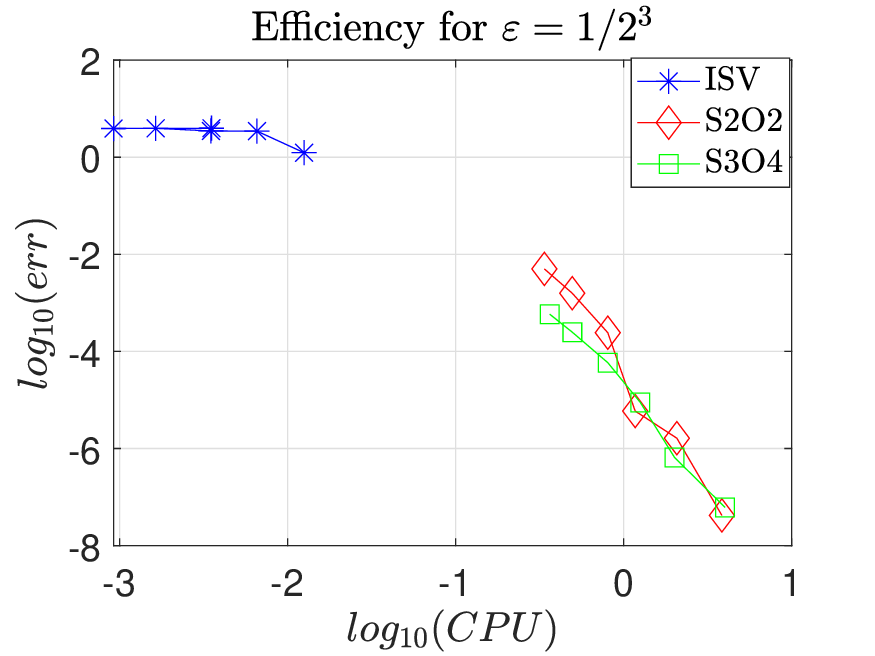}
\caption{Efficiency: the plot of the error $err=err_u+err_v$  against the CPU time (CPU)  with $\hh=1/2^k$ for $k=1,2,\ldots,6$.} \label{p10}
\end{figure}

{\subsubsection{2D test}
We now turn to the 2D case with $\lambda = 1$, employing the initial data defined in \cite{zhao19}:
 \begin{align}
 \varphi_{1}(x,y) &= \exp(-(x+2)^{2} - y^{2}) + \exp(-(x-2)^{2} - y^{2}), \\
 \varphi_{2}(x,y) &= \exp(-x^{2} - y^{2}), \quad (x,y) \in (-16,16)^{2}.
 \end{align}
 The numerical parameters are set as $N_x = N_y = 2^8$ and $N_\tau = 2^5$. The problem is first solved for $\eps=0.05$ with periodic boundary conditions by the S3O4 scheme, and the resulting solution contours are shown in Figure \ref{p2d1}. Subsequently, the accuracy and energy conservation properties are illustrated in Figures \ref{p2d2} and \ref{p2d3}. All observations are consistent with the 1D case.}
   \begin{figure}[t!]
\centering
\includegraphics[width=4.5cm,height=4.5cm]{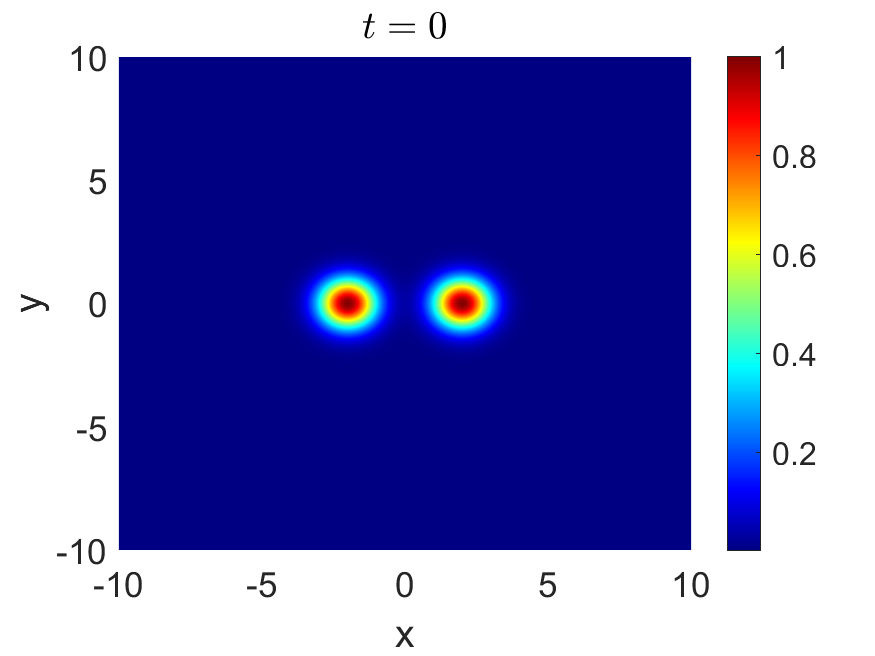}
\includegraphics[width=4.5cm,height=4.5cm]{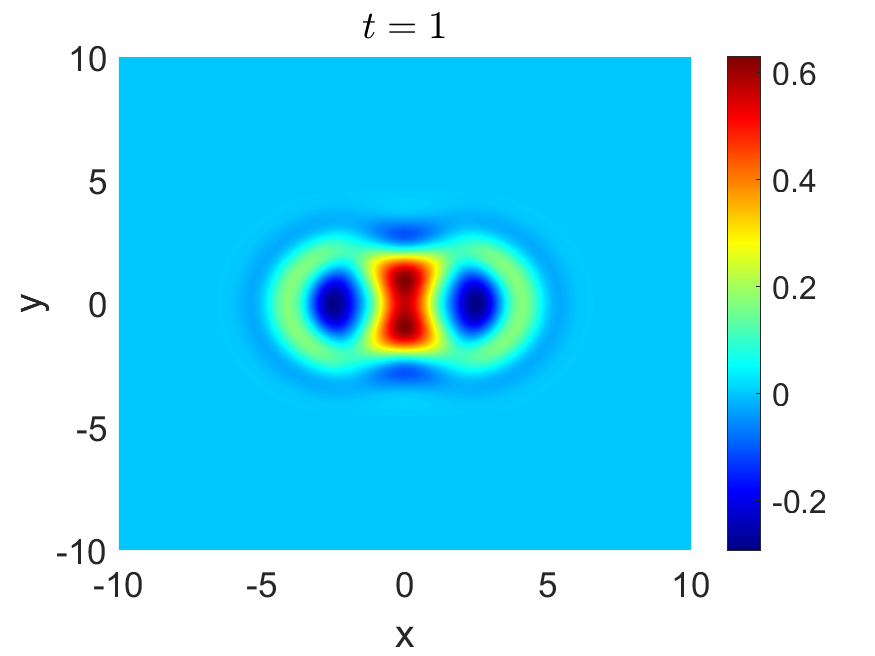}
\includegraphics[width=4.5cm,height=4.5cm]{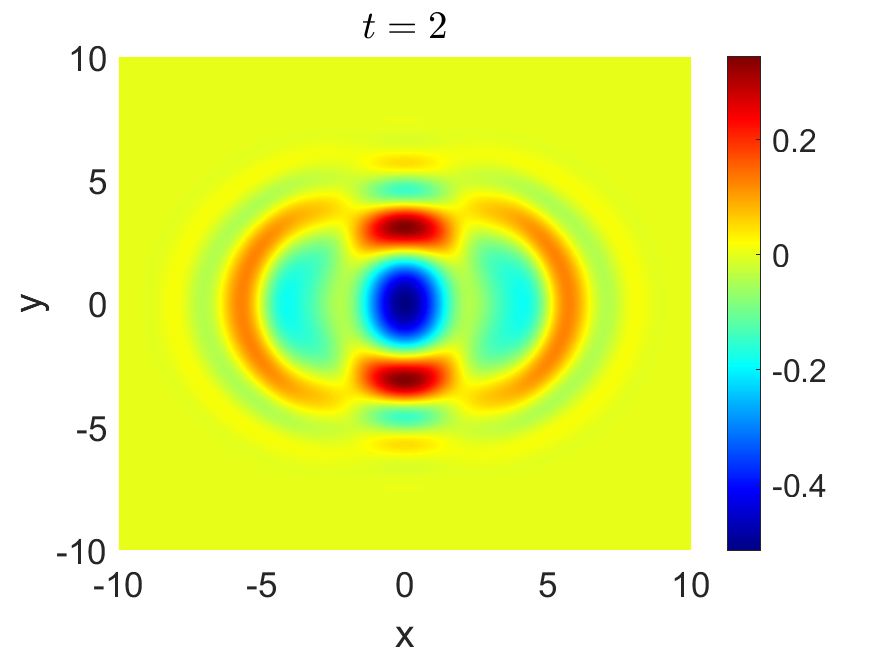}
\caption{2D test: contour plots of the solutions obtained by S3O4  with $h=0.1$ at different time $t$.} \label{p2d1}
\end{figure}

   \begin{figure}[t!]
\centering
\includegraphics[width=5.5cm,height=4.8cm]{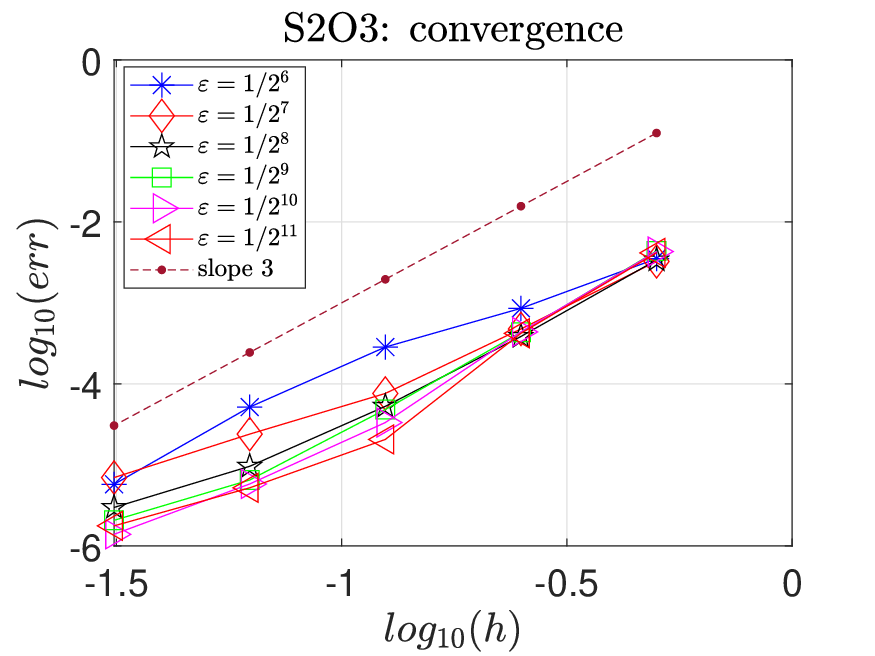}
\includegraphics[width=5.5cm,height=4.8cm]{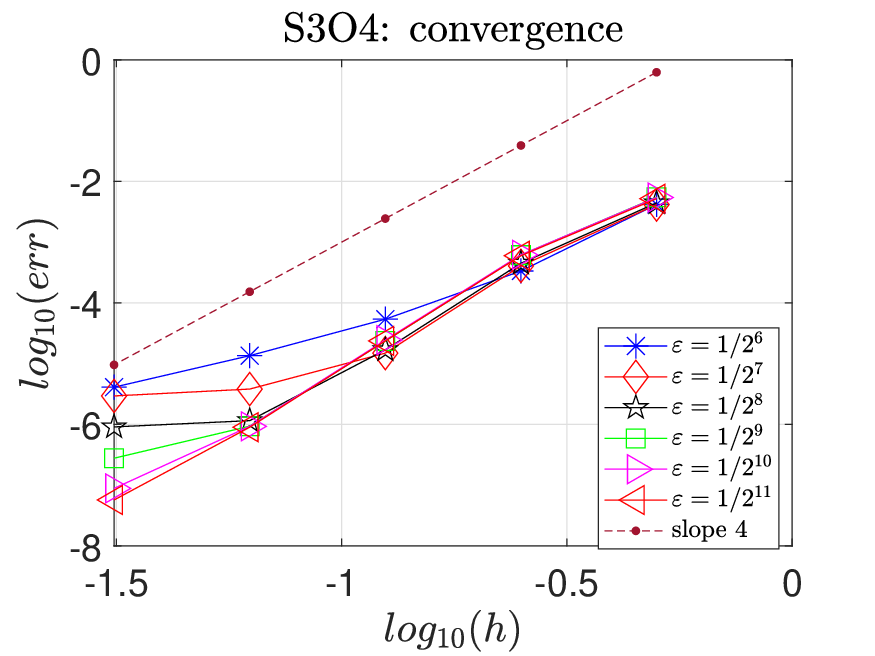}
\caption{2D test:  the log-log plot of the temporal errors $err$  at $ t _n=1$ under different $\eps$, where $\hh=1/2^k$ with $k=1,2,\ldots,5$.} \label{p2d2}
\end{figure}

   \begin{figure}[t!]
\centering
\includegraphics[width=5.5cm,height=4.8cm]{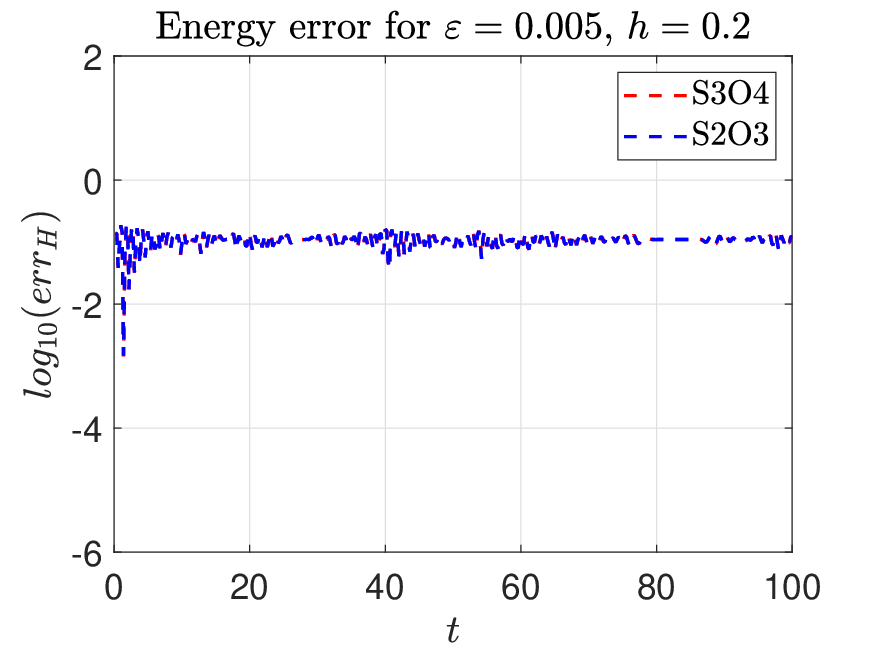}
\includegraphics[width=5.5cm,height=4.8cm]{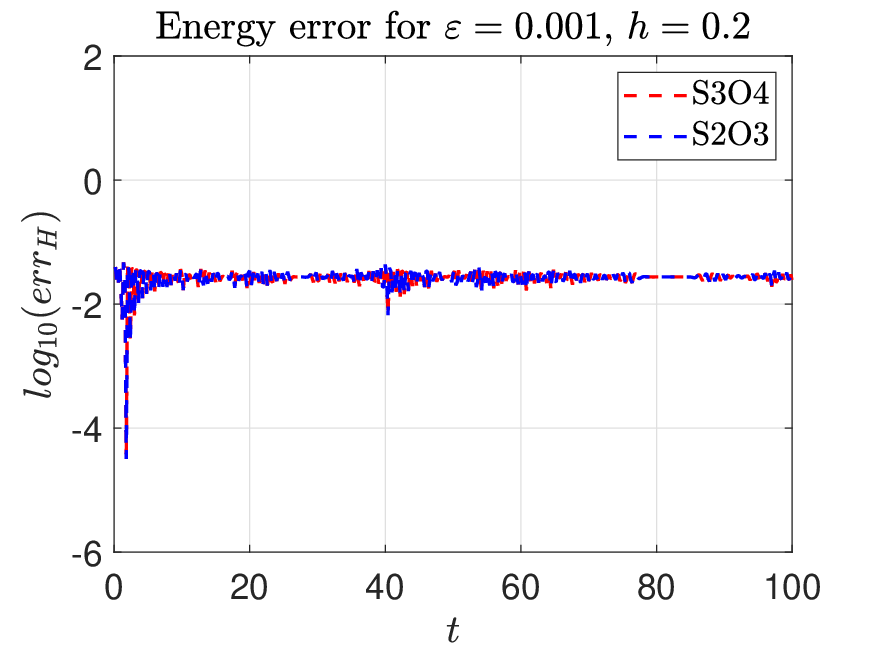}
\caption{2D test:  the plot of the energy error $err_H$ against $t$.} \label{p2d3}
\end{figure}

\section{Proof of  high accuracy  (Theorem \ref{UA thm})}\label{sec:4}
The convergence of the three-stage integrator of order four (S3O4 with $s=3$) is studied in this section. The proof is easily presented for the third order S2O3  and we omit it  for brevity.

\begin{proof}The proof is presented in four steps.

\textbf{I. A lemma and its proof.} A preliminary result is needed in the analysis and we present it by the following lemma.
\begin{lemma}\label{add lem}
Under the conditions of Proposition \ref{TS M},
 for any $\kappa>0$, there exists  $0<T_{\kappa}\leq T$ such that for all   $0<\eps\leq \eps_0$, the  two-scale system \eqref{2scale compact} with the  initial condition $X_0$ \eqref{inv} has a unique solution
$X(\cdot, t ,\tau)\in C^0([0,T_{\kappa} ]\times\mathbb{T};H^{\nu})$ with the bound
\begin{equation}\label{bound 2scaleNEW}
 \sup_{  0<\eps\leq \eps_0}\norm{X(\cdot, t ,\cdot)}_{L^{\infty}_{\tau}(H^{\nu})}\leq \kappa
 \sup_{  0<\eps\leq \eps_0}\norm{X_0(\cdot)}_{L^{\infty}_{\tau}(H^{\nu})}  \forall\      t \in[0,T_{\kappa}  ].
\end{equation}
In addition, the solution $X(\cdot, t ,\tau)$ has first derivatives w.r.t. both $\tau$ and $ t $ and they satisfy
 $$\partial_\tau X(\cdot, t ,\tau)\in C^0([0,T_{\kappa}]\times\mathbb{T};H^{\nu-1}),\ \ \ \partial_{ t }  X(\cdot, t ,\tau)\in C^0([0,T_{\kappa}]\times\mathbb{T};H^{\nu-1}).$$
Moreover, if the nonlinear function $G(t,\tau,X)$ satisfies
$$\norm{G(t,\tau,X)}_{H^{\nu}}\leq C_1\norm{X}_{H^{\nu}}+C_2$$
for some positive constants $C_1,C_2$, then the unique solution $$X(\cdot, t ,\tau)\in C^0([0,T]\times\mathbb{T};H^{\nu})$$ of  \eqref{2scale compact} further has the following estimate
$$ \sup_{  0<\eps\leq \eps_0}\norm{X(\cdot, t ,\cdot)}_{L^{\infty}_{\tau}(H^{\nu})}\leq \Big(
 \sup_{  0<\eps\leq \eps_0}\norm{X_0(\cdot)}_{L^{\infty}_{\tau}(H^{\nu})}+ C_2  t    \Big)e^{C_1   t  }$$
 for $  t \in[0,T].$
\end{lemma}
\begin{proof}This lemma can be shown in a similar way as Proposition 2.1 of \cite{Chartier15} with some modifications.
Denote the smooth solution $X(\cdot, t ,\tau)$ of \eqref{2scale compact} by a new function $$\chi( t ,\tau)=X(\cdot, t ,\tau+ t /\eps^2).$$ Based on
\eqref{2scale compact}, it can be seen that
$$\partial_{ t }\chi( t ,\tau)= G(t,\tau+ t /\eps^2,\chi( t ,\tau))$$ with $\chi(0,\tau)=X_0(\tau).$
Hence $\partial_{ t } \chi$ is in $H^{\nu}$.
From the Cauchy–Lipschitz theorem in  $H^{\nu}$, it follows that this system has a unique solution on $[0,T]$ which satisfies
$$\norm{\chi( t ,\tau)}_{H^{\nu}}\leq C_1\norm{X_0(\tau)}_{H^{\nu}}+\int_0^{ t } \norm{G(t,\tau+\theta/\eps^2,\chi(\theta,\tau))}_{H^{\nu}}d \theta.$$
Based on this result and with the same arguments of \cite{Chartier15}, we obtain the estimate \eqref{bound 2scaleNEW}.

Now we derive the differential equation of $\partial_\tau \chi ( t ,\tau)$ which reads
  \begin{equation}\label{ptau de}\begin{aligned}&\partial_{ t } \big(\partial_\tau \chi\big) ( t ,\tau)=\partial_{\tau} \big(\partial_{ t } \chi\big) ( t ,\tau)\\
  =& \partial_{\tau} G(t,\tau+ t /\eps^2,\chi( t ,\tau))+ \partial_{X} G(t,\tau+ t /\eps^2,\chi( t ,\tau))\partial_\tau \chi( t ,\tau).\end{aligned}  \end{equation}
Observing that the initial value $\partial_\tau \chi (0,\tau)=\partial_\tau X_0(\tau)\in  H^{\nu-1}$ and $\partial_{\tau} G$ (resp. $\partial_{X} G$) is a continuous and locally bounded function  from   $\bT\times H^{\nu}$  to $H^{\nu-1}$ (resp.  $\mathcal{L}(H^{\nu-1},H^{\nu-1})$), one gets that the solution $\partial_\tau \chi$ of  differential equation  \eqref{ptau de} is unique and  in $H^{\nu-1}$. Moreover, according to $\chi( t ,\tau)=X(\cdot, t ,\tau+ t /\eps^2)$, it is deduced that
$$\partial_{\tau}X(\cdot, t ,\tau)=\partial_{\tau} \chi( t ,\tau- t /\eps^2) \in H^{\nu-1}$$
and  $$\partial_{ t }X(\cdot, t ,\tau)=\partial_{ t } \chi( t ,\tau- t /\eps^2)-\frac{1}{\eps^2}\partial_{\tau} \chi( t ,\tau- t /\eps^2)  \in H^{\nu-1}.$$
The last result can be done  using the bootstrap type argument and the Gronwall lemma.
\end{proof}

\textbf{II. Proof of Proposition  \ref{TS M} (boundedness of the solution of the two-scale system).}
Noticing that the operators $\Pi $ and $A$ are both bounded on $C^0(\bT;H^{\sigma})$ for any $\sigma\geq0$, the initial value  \eqref{inv}  can be estimated that $X_0$ \eqref{inv} is uniformly bounded w.r.t. $\eps$ in $L^{\infty}_{\tau}(H^{\nu})$. Combining this with Lemma \ref{add lem},
 it is known that the two-scale system  \eqref{2scale}  as well as the initial value \eqref{inv}    determines a unique solution satisfying $\norm{X( t ,\tau)}_{L^{\infty}_{\tau}(H^{\nu})}\leq C$. Furthermore, the first derivative of $X( t ,\tau)$\footnote{In the following parts of this section, we omit the expression $x$ for bervity.} w.r.t. $ t $  exists and is a function with value in $H^{\nu-1}$.

 In what follows, we study the boundedness of the  derivatives of $X( t ,\tau)$ w.r.t. $t$ and hope to get a more rigorous result. To this end, we first consider the differential equation of the first derivative  $\partial_{ t } X( t ,\tau):= V( t ,\tau)$ which reads
 $$\partial_{ t } V( t ,\tau)+\frac{1}{\eps^2}
 \partial_\tau V( t ,\tau)=\partial_{t}  G(t ,\tau,X( t ,\tau))+\partial_{X}  G(t ,\tau,X( t ,\tau))V( t ,\tau).$$
The initial date of this system is
  \[
\begin{aligned}
&V_0:=V(0,\tau)=\partial_{ t } X(0,\tau)= G(t,\tau,X_0)-\frac{1}{\eps^2}L
X_0 \\
=& G(t,\tau,X_0)- L \kappa_1(\tau,\underline{X}^{[0]})-\eps^2  L\kappa_2(\tau,\underline{X}^{[0]})-\eps^4 L\kappa_3(\tau,\underline{X}^{[0]})\\
=& G(t,\tau,X_0)- G(t,\tau,\underline{X}^{[0]})+\Pi G(t,\tau,\underline{X}^{[0]})-\eps^2  L\kappa_2(\tau,\underline{X}^{[0]})-\eps^4 L\kappa_3(\tau,\underline{X}^{[0]})\\
=& \partial_{X} G(t,\tau,\underline{X}^{[0]}) (X_0-\underline{X}^{[0]})+ \Pi G(t,\tau,\underline{X}^{[0]})+\mathcal{O}_{\mathcal{X}^{\nu-1}}(\eps^{2})\\
=& \eps^2 \partial_{X} G(t,\tau,\underline{X}^{[0]}) \kappa_1(\tau,\underline{X}^{[0]})+ \Pi G(t,\tau,\underline{X}^{[0]})+\mathcal{O}_{\mathcal{X}^{\nu-1}}(\eps^{2})
\\
=&  \Pi G(t,\tau,\underline{X}^{[0]})+\mathcal{O}_{\mathcal{X}^{\nu-1}}(\eps^{2}).
 \end{aligned}
\]

According to this result, we obtain $$V_0=\mathcal{O}_{\mathcal{X}^{\nu-1}}(1).$$ Moreover, the nonlinear function satisfies
$$\norm{\partial_{t}  G(t ,\tau,X( t ,\tau))+\partial_{X}  G(t ,\tau,X( t ,\tau))V( t ,\tau)}_{H^{\nu-1}}\leq  C_1+C_2\norm{V( t ,\tau)}_{H^{\nu-1}}.$$
 These results and Lemma \ref{add lem}  yield $$\partial_{ t } X( t ,\tau)=\mathcal{O}_{\mathcal{X}^{\nu-1}}(1).$$

{This procedure can be proceeded in an analogous way for $$\partial^2_{ t } X( t ,\tau),\partial^3_{ t } X( t ,\tau),\partial^4_{ t } X( t ,\tau),$$ and we get
  $$\partial^2_{ t } X( t ,\tau)=\mathcal{O}_{\mathcal{X}^{\nu-2}}(1),\ \ \partial^3_{ t } X( t ,\tau)=\mathcal{O}_{\mathcal{X}^{\nu-3}}(1),\ \ \partial^4_{ t } X( t ,\tau)=\mathcal{O}_{\mathcal{X}^{\nu-4}}(1).$$}

\textbf{III. Proof of Proposition  \ref{stiff thm} (local errors of exponential integrators).} Then we prove Proposition  \ref{stiff thm}, where the local errors of our integrators for the transformed system are stated.

 For all  $\hh\geq0$, it is true that $\varphi_0(\hh M)$ and
$\varphi_i(\hh M)$ for  $i=1,2,3$ are   uniformly bounded w.r.t. $\eps$. Therefore,
the coefficients $\bar{a}_{i\rho}(\hh M), \bar{b}_{\rho}(\hh M) \ \textmd{for}\  i,\rho=1,2,3$ of exponential integrators are uniformly  bounded.

Define the error functions by
\[
\begin{aligned}
&e^n(\tau):=Z( t _n,\tau)-I_{\mathcal{M}} Z^n,\ \ E^{ni}(\tau):=Z( t _n+c_i\hh ,\tau)-I_{\mathcal{M}} Z^{ni},\end{aligned}
\]
and the projected errors as
\[
\begin{aligned}
&e_{\mathcal{M}}^n(\tau):=P_{\mathcal{M}} Z( t _n,\tau)-Z_{\mathcal{M}}^n,\ \ E_{\mathcal{M}}^{ni}(\tau):=P_{\mathcal{M}} Z( t _n+c_i\hh ,\tau)-Z_{\mathcal{M}}^{ni},\end{aligned}
\]
where $I_{\mathcal{M}} $ and $P_{\mathcal{M}} $ are defined in  \eqref{PI}.
 By the
triangle inequality and estimates on projection error \cite{Shen}, one has
\[
\begin{aligned}  \|e^n\|_{H^{\nu-4}}&\leq  \|e_{\mathcal{M}}^n\|_{H^{\nu-4}}+ \| Z_{\mathcal{M}}^n-I_{\mathcal{M}} Z^n\|_{H^{\nu-4}}+ \|Z( t _n,\tau)-P_{\mathcal{M}} Z( t _n,\tau)\|_{H^{\nu-4}} \\&\lesssim \|e_{\mathcal{M}}^n\|_{H^{\nu-4}} + \delta^{\tau}_{\mathcal{F}},
\end{aligned}
\]
and similarly $\|E^{ni}\|_{H^{\nu-4}}\lesssim \|E_{\mathcal{M}}^{ni}\|_{H^{\nu-4}}+ \delta^{\tau}_{\mathcal{F}}$, where $\delta^{\tau}_{\mathcal{F}}$  denotes the  error in   $\tau$ brought by the Fourier pseudo-spectral method.
Therefore, the estimations for $e^n$ and  $E^{ni}$ could be converted to the estimations for  $e_{\mathcal{M}}^n$ and $E_{\mathcal{M}}^{ni}$.

Since the scheme \eqref{erk dingyi} is implicit, iterative solutions
are {needed,} and  we consider
pattern
\begin{equation*}
\begin{aligned}
&\big(\widetilde{Z^{ni}}\big)^{[0]}=e^{c_{i}\hh M}\widetilde{Z^{n}}+  \hh \textstyle\sum\limits_{\rho=1}^{3}\bar{a}_{i\rho}(\hh M) \Gamma  \big(t_n+c_\rho h,\widetilde{Z^{n}}\big),\\
&\big(\widetilde{Z^{ni}}\big)^{[j+1]}=e^{c_{i}\hh M}\widetilde{Z^{n}}+ \hh \textstyle\sum\limits_{\rho=1}^{3}\bar{a}_{i\rho}(\hh M) \Gamma  \big(t_n+c_\rho h,\big(\widetilde{Z^{n\rho}}\big)^{[j]}\big),
\end{aligned}
\end{equation*}
for
$j=0,1,\ldots,j^\textmd{stopped}.$
 Then according to the property of $e^{c_{i}\hh M}$ and the
boundedness of the coefficients, the following result can be proved.
That is, there exists a small constant $0<\hh_0< 1$ such that when
$0<\hh \leq \hh _0$ and $\widetilde{Z^0} \in H^{\nu}$ with $ \|\widetilde{Z^0}
\|_{H^{\nu}}\leq K_1$,   we  can obtain  $\widetilde{Z^{n}}\in H^{\nu},\
\ \big(\widetilde{Z^{ni}}\big)^{[j^\textmd{stopped}]}\in H^{\nu} $ as well as
their bounds $\norm{\widetilde{Z^{n}}}_{H^{\nu}}\leq C_0,\ \
\norm{\big(\widetilde{Z^{ni}}\big)^{[j^\textmd{stopped}]}}_{H^{\nu}}\leq C_0, $
where the constant $C_0$ depends on $T$ and $K_1$.

The error system of FS-F method is to find
$e_{\mathcal{M}}^n(\tau)$ and
$E_{\mathcal{M}}^{ni}(\tau)$ in the
space $Y_{\mathcal{M}}$, {i.e., we have}
  \begin{equation*}
\begin{aligned}
 &e_{\mathcal{M}}^{n+1}(\tau)=\sum\limits_{k\in \mathcal{M}}
\big(\widetilde{e_{\mathcal{M}}^{n+1}}\big)_k\mathrm{e}^{\mathrm{i} k \tau },\quad
E_{\mathcal{M}}^{ni}(\tau)=\sum\limits_{k\in \mathcal{M}}
\big(\widetilde{E_{\mathcal{M}}^{ni}}\big)_k\mathrm{e}^{\mathrm{i} k \tau },
\end{aligned}
\end{equation*}
where
\begin{equation}\label{error equ}
\begin{aligned}
\widetilde{e_{\mathcal{M}}^{n+1}}&=e^{\hh M}\widetilde{e_{\mathcal{M}}^{n}}+\hh \sum\limits_{\rho=1}^{3}\bar{b}_{\rho}(\hh M) \Delta\widetilde{\Gamma^{n\rho}}+\widetilde{\delta^{n+1}},\\
\widetilde{E_{\mathcal{M}}^{ni}}&=e^{c_i\hh M}\widetilde{e_{\mathcal{M}}^{n}}+ \hh   \textstyle\sum\limits_{\rho=1}^{3}\bar{a}_{i\rho}(\hh M) \Delta\widetilde{\Gamma^{n\rho}}+\widetilde{\Delta^{ni}},
\end{aligned}
\end{equation}
{and} $\Delta\widetilde{\Gamma^{n\rho}}:=\Gamma  \big( t _n+c_{\rho}\hh, P_{\mathcal{M}} Z( t _n+c_{\rho}\hh ,\tau)\big)- \Gamma  \big( t _n+c_{\rho}\hh,Z_{\mathcal{M}}^{n\rho}\big)$ . Here the remainders
$\widetilde{\delta^{n+1}},\ \widetilde{\Delta^{ni}}$   are
bounded by  inserting the exact solution of \eqref{2scale Fourier}
into the numerical approximation, {i.e.,}
\begin{equation*}
\begin{array}[c]{ll}%
\widetilde{\mathbf{Z}( t _n+\hh )}&=e^{\hh M}\widetilde{\mathbf{Z}( t _n)}+   \hh \textstyle\sum\limits_{\rho=1}^{3}\bar{b}_{\rho}(\hh M)
\Gamma  \big( t _n+c_{\rho}\hh,\widetilde{\mathbf{Z}( t _n+c_{\rho}\hh )}\big)+\widetilde{\delta^{n+1}},\\
\widetilde{\mathbf{Z}( t _n+c_i\hh )}&=e^{c_{i}\hh M}\widetilde{\mathbf{Z}( t _n)}+  \hh \textstyle\sum\limits_{\rho=1}^{3}\bar{a}_{i\rho}(\hh M) \Gamma  \big( t _n+c_{\rho}\hh,\widetilde{\mathbf{Z}( t _n+c_{\rho}\hh )}\big)+\widetilde{\Delta^{ni}}.
\end{array}
\end{equation*}
By the Duhamel principle and Taylor expansions, the remainders $\widetilde{\delta^{n+1}}$   can be represented as
\[
\begin{aligned}  \widetilde{\delta^{n+1}}=& \hh    \int_{0}^1   e^{(1-z)\hh M}  \sum\limits_{\rho=1}^{4}\frac{(z  \hh )^{\rho-1}}{(\rho-1)!}\frac{\textmd{d}^{\rho-1}}{\textmd{d}  t ^{\rho-1}} \widetilde{\Gamma( t _n) } {\rm d}z\\&+\hh  \sum\limits_{j =1}^{s}   \bar{b}_{j}(\hh M) \sum\limits_{\rho=1}^{4}\frac{ c_{j}^{\rho-1}\hh^{\rho-1}}{(\rho-1)!}\frac{\textmd{d}^{\rho-1}}{\textmd{d} t ^{\rho-1}}\widetilde{\Gamma( t _n)}+\widetilde{\delta_{4}^{n+1}} \\
=&   \sum\limits_{\rho=1}^{4}\hh^\rho \psi_{\rho}(\hh M)  \frac{\textmd{d}^{\rho-1}}{\textmd{d} t ^{\rho-1}}\widetilde{\Gamma( t _n)}+\widetilde{\delta_{4}^{n+1}},
\end{aligned}
\]
where  we take the notation
$\widetilde{\Gamma( t )}:=\Gamma  \big(t,\widetilde{\mathbf{Z}( t )}\big)$.
With the bounds of the solution of the two-scale system proposed in  Proposition  \ref{TS M}, it is obtained that
$\frac{\textmd{d}^{\rho}}{\textmd{d} t ^{\rho}}\widetilde{\Gamma( t _n)}=\mathcal{O}_{H^{\nu-\rho}}(1)$ for $\rho=0,1,\ldots,4$.
Thus  we get for $\zeta\in[0,1]$ $$\widetilde{\delta^{n+1}}=\mathcal{O}_{H^{\nu-4}}\Big(   \hh ^{5}\frac{\textmd{d}^{4}}{\textmd{d} t^{4}} \widetilde{\Gamma( t _n+\zeta \hh ) } \Big)=\mathcal{O}_{H^{\nu-4}}\big(  \hh^{5} \big).$$
Similarly, one has
$$\widetilde{\Delta^{ni}}=    \sum\limits_{\rho=1}^{3}\hh^{\rho} \psi_{\rho,i}(\hh M)   \frac{\textmd{d}^{\rho-1}}{\textmd{d}  t ^{\rho-1}}\widetilde{\Gamma( t _n)}+\widetilde{\Delta_{3}^{ni}}
$$
with $\widetilde{\Delta_{3}^{ni}}=\mathcal{O}_{H^{\nu-3}}\big(  \hh ^{4} \big)$.
Then using the stiff order conditions presented in  Proposition \ref{stiff
thm},
{we know that} $ \widetilde{\Delta^{ni}}=\widetilde{\Delta_{3}^{ni}}$ and $$\widetilde{\delta^{n+1}}=    \hh^4 \psi_{4}(\hh M) \frac{\textmd{d}^{3}}{\textmd{d} t ^{3}}\widetilde{\Gamma( t _n)}+\widetilde{\delta_{4}^{n+1}}.$$ The
proof of Proposition  \ref{stiff thm} is immediately complete.

\textbf{IV. Proof of the global errors.}
{We are now in a position} to derive the error bounds in a standard way. To make the analysis be more compact, here we only present  the main points but without details.

For the  error recursion \eqref{error equ}, by Taylor series, one gets $\Delta\widetilde{\Gamma^{n\rho}}=J_n\widetilde{E_{\mathcal{M}}^{n\rho}}$ with a matrix $J_n$.
Then there exist bounded operators $\mathcal{N}^{ni}(\widetilde{e_{\mathcal{M}}^{n}})$ such that
$$\widetilde{E_{\mathcal{M}}^{ni}}=\mathcal{N}^{ni}(\widetilde{e_{\mathcal{M}}^{n}})\widetilde{e_{\mathcal{M}}^{n}}
+\widetilde{\Delta^{ni}}+\hh^4 \mathcal{R}^{ni}$$ with uniformly bounded remainders $\mathcal{R}^{ni}$ in $H^{\nu-4}$.
Now it is arrived at
\begin{equation*}
\begin{aligned}
\widetilde{e_{\mathcal{M}}^{n+1}}=&e^{\hh M}\widetilde{e_{\mathcal{M}}^{n}}+\hh   \sum\limits_{\rho=1}^{3} \bar{b}_{\rho}(\hh M) J_n\mathcal{N}^{n\rho}(\widetilde{e_{\mathcal{M}}^{n}})\widetilde{e_{\mathcal{M}}^{n}}\\
&+\hh   \sum\limits_{\rho=1}^{3} \bar{b}_{\rho}(\hh M)J_n\widetilde{\Delta^{n\rho}} +   \hh^4 \psi_{4}(\hh M)\frac{\textmd{d}^{3}}{\textmd{d}  t ^{3}}\widetilde{\Gamma( t _n)}+\mathcal{O}_{H^{\nu-4}}\big(  \hh^{5}\big)\\
=& e^{\hh M}\widetilde{e_{\mathcal{M}}^{n}}+\hh  \mathcal{N}^{n}(\widetilde{e_{\mathcal{M}}^{n}}) \widetilde{e_{\mathcal{M}}^{n}} +   \hh^4 \psi_{4}(\hh M)\frac{\textmd{d}^{3}}{\textmd{d} t ^{3}}\widetilde{\Gamma( t _n)}+\mathcal{O}_{H^{\nu-4}}\big(  \hh^{5}\big),
\end{aligned}\end{equation*}
where we use the notation $\mathcal{N}^{n}(\widetilde{e_{\mathcal{M}}^{n}}):=\sum\limits_{\rho=1}^{3} \bar{b}_{\rho}(\hh M) J_n\mathcal{N}^{n\rho}(\widetilde{e_{\mathcal{M}}^{n}}).$
Solving this recursion leads to
\begin{equation*}
\begin{aligned}
\widetilde{e_{\mathcal{M}}^{n}}=&\hh  \sum_{j=0}^{n-1} e^{(n-j-1)\hh M} \mathcal{N}^{j}(\widetilde{e_{\mathcal{M}}^{j}}) \widetilde{e_{\mathcal{M}}^{j}}\\&+   \hh^4\sum_{j=0}^{n-1} e^{(n-j-1)\hh M}\psi_{4}(\hh M)  \frac{\textmd{d}^{3}}{\textmd{d} t ^{3}}\widetilde{\Gamma( t _n)}+\mathcal{O}_{H^{\nu-4}}\big(\hh ^{4} \big).
\end{aligned}\end{equation*}
The order condition $\psi_{4}(0)=0$ shows that there exists bounded operator $\tilde{\psi}_{4}(\hh M)$ with
 $\psi_{4}(\hh M)=\hh M \tilde{\psi}_{4}(\hh M)$ and Lemma 4.8 of \cite{Ostermann06} contributes %
\begin{equation*}\begin{array}{ll}\sum_{j=0}^{n-1} e^{(n-j-1)\hh M}\psi_{4}(\hh M) \frac{\textmd{d}^{3}}{\textmd{d} t ^{3}}\widetilde{\Gamma( t _n)}
&=\sum_{j=0}^{n-1}e^{(n-j-1)\hh M}\hh M \tilde{\psi}_{4}(\hh M)  \frac{\textmd{d}^{3}}{\textmd{d} t ^{3}}\widetilde{\Gamma( t _n)}\\
&=\mathcal{O}_{H^{\nu-4}}\big(\frac{\textmd{d}^{3}}{\textmd{d}  t ^{3}}\widetilde{\Gamma( t _n)}\big) =\mathcal{O}_{H^{\nu-4}} (1).
\end{array}
\end{equation*}
  Combined with the above results and using Gronwall inequality leads to $ \widetilde{e_{\mathcal{M}}^{n}}=\mathcal{O}_{H^{\nu-4}}\big(\hh ^{4}  \big).$
  This result and the formulation of the scheme give the final convergence for the numerical solutions of the original system \eqref{klein-gordon}:
  \begin{equation*}
\begin{aligned}
 &\norm{u(\cdot,n\hh)-u^n}_{H^{\nu-3}} \leq\norm{X(\cdot,n\hh,n\hh/\eps^2)-I_{\mathcal{M}} Z^n}_{H^{\nu-4}} \\ \leq&\norm{X(\cdot,n\hh,n\hh/\eps^2)-Z( t ,\tau)}_{H^{\nu-4}}+\norm{Z( t ,\tau)-I_{\mathcal{M}} Z^n}_{H^{\nu-4}}\\ \leq&\delta^x_{\mathcal{F}}+\norm{P_{\mathcal{M}} Z( t _n,\tau)-Z_{\mathcal{M}}^n }_{H^{\nu-4}}+\delta^{\tau}_{\mathcal{F}}
   \leq\delta_{\mathcal{F}}+  C  \hh^4,\end{aligned}
\end{equation*}
 and similarly
    \begin{equation*}
\begin{aligned}
 \norm{v(\cdot,n\hh)-v^n}_{H^{\nu-4}} \leq\delta^x_{\mathcal{F}}+\frac{1}{\eps^2}\norm{P_{\mathcal{M}} Z( t _n,\tau)-Z_{\mathcal{M}}^n }_{H^{\nu-3}}+\delta^{\tau}_{\mathcal{F}}
   \leq\delta_{\mathcal{F}}+  C  \hh^4/\eps^{2}.
\end{aligned}
\end{equation*}

The proof  of Theorem \ref{UA thm} is complete.
\end{proof}
%
%

\section{Proof of long time energy {near conservation} (Theorem  \ref{Long-time thm})}\label{sec:5}
Long time energy {near conservation} is proved mainly based on the technology of modulated Fourier expansions, which
  was firstly developed  in \cite{hairer2000} and was used for the long-term analysis of many methods \cite{Cohen0,Cohen2,Lubich2006,Lubich2,lubich19,WW,WZ,WZ21}. The main differences and contributions of long term analysis given in this section involve in two aspects. We extend  the technology of modulated Fourier expansions to multi-stage schemes and prove the long-time result  for  both two-stage   and three-stage methods. This  provides some developments for studying long-term behavior of various methods.   Moreover, in contrast to
the existing work, we neither assume bounded energy, nor assume small initial value for the considered system. In the proof,
the result is derived for the methods applied to the energy unbounded system \eqref{klein-gordon} with large initial value.
 Finite dimensional vector space and Euclidean norm are considered in the analysis.

\begin{proof}The proof consists of three parts. The first two parts are given for S2O3 and the last one is devoted to the proof of S3O4.

\textbf{I. Modulated Fourier expansions.}
\begin{lemma}
With the notations introduced in Section \ref{sec:31} and under the conditions given in Theorem \ref{Long-time thm},  the numerical result $ \widehat{Z^{n}}$  \eqref{cs ei-ful2} obtained from
S2O3  can be expressed by the
following modulated Fourier expansion at   $  t =n\hh$
{\begin{equation}\begin{array}{ll}
\widehat{Z^{n}}=\sum\limits_{\mathbf{k}\in\mathcal{N}_N^*}
\mathrm{e}^{\mathrm{i}(\mathbf{k} \cdot \boldsymbol{\omega})
 t }\alpha^{\mathbf{k}}( t )+R_{Z ,N}( t ),
           \end{array} \label{MFE-ERKN-0}%
\end{equation}
where  $\alpha^{\mathbf{k}}$   are smooth coefficient functions for $  t =n\hh$. These functions and their derivatives  have the  bounds
\begin{equation}%
\begin{array}{ll}
  \alpha_{j,l}^{\langle j\rangle_l}( t )=\mathcal{O}(\delta_0),\ \
  \dot{\alpha}_{j,l}^{\langle j\rangle_l}( t )=\mathcal{O}(\frac{\hh^2 }{\iota_1 \norm{\iota_2}_{L^2}}  \delta_0^3),\\
 \alpha_{j,l}^{\mathbf{k}}( t )=\mathcal{O}\Big(\frac{\hh   }{ \norm{\iota_2}_{L^2}}   \delta_0^{|\mathbf{k}|} \Big),\ \    {\mathbf{k}\neq \langle j\rangle_l,}
\end{array} %
\label{coefficient func}%
\end{equation}
where $j\in \{-N_\tau/2,-N_\tau/2+1,\ldots,N_\tau/2\}$,  $l=1,2,\ldots,2\hat{d}$ and \begin{equation}\label{iota}%
\begin{array}{ll}&\iota_1:=\frac{\hh^3 \omega_{j,l}^2}{4\sin^2(\hh \omega_{j,l}/2)},\   \iota_2 :=- \mathrm{i}\hh^2 \Omega^2  \Big(\hh  \Omega-
\cot\zeta+
\cos\zeta\csc\zeta\Big)^{-1},
\end{array} %
\end{equation}
with $\zeta:= \frac{1}{2}\hh (\Omega -(\mathbf{k}\cdot \boldsymbol{\omega})I) $.
The remainder   appeared in  \eqref{MFE-ERKN-0} is bounded by
\begin{equation}\begin{array}{ll}
&R_{Z ,N}( t )=\mathcal{O}( t    h^2\delta_0^{N+1}),
\end{array}\label{remainder}%
\end{equation}}The constants symbolised by the notation $\mathcal{O}$ are independent of $\hh, \eps$, but depend on $ c_0, c_1$ appeared in the conditions of Theorem \ref{Long-time thm}.
\end{lemma}

\begin{proof} \textbf{Proof of \eqref{MFE-ERKN-0}.}
 We first derive the following modulated Fourier expansion of the numerical integrators, i.e.,  the numerical scheme S2O3 has the formal expansion:
\begin{equation}
\begin{array}{ll}
\widehat{Z^{n}}=\Phi( t ):= \sum\limits_{\mathbf{k}\in\mathcal{N}^*_{\infty}}
\mathrm{e}^{\mathrm{i}(\mathbf{k} \cdot\boldsymbol{\omega} )
 t }\alpha^{\mathbf{k}}( t ),
\end{array} \label{MFE-1}%
\end{equation}
where $\mathcal{N}^*_{\infty}$ denotes the set of $\mathcal{N}_N^*$ with $N=+\infty$.
With the same arguments of local errors, it can be shown that the error between
$\widehat{Z^{ni}}$ and $\Phi( t +c_i\hh )$ has the form $\mathcal{O}(  \hh^{2})\ddot{\Phi}( t +\theta \hh)$ for some $\theta\in [0,c_i]$. Hence, one can assume that for $i=1,2$
    \begin{equation*} \begin{array}[c]{ll}%
    \widehat{Z^{ni}}=\Phi( t +c_i\hh )+C  \hh^2\mathcal{D}^2\Phi( t +c_i\hh ),
\end{array}\end{equation*}
where $C$ is the error constant which is independent of $\hh,\eps$ and $\mathcal{D}$ is referred to the differential operator (see \cite{hairer2006}) .

%
Inserting  \eqref{MFE-1}   into \eqref{cs ei-ful2} and
defining the operators  \begin{equation*}\begin{array}{ll}\mathcal{L}_1(\hh \mathcal{D}):=& \big(e^{\hh \mathcal{D}} -e^{\hh M } \big)
\big(\bar{b}_{1}(\hh M)e^{c_1\hh \mathcal{D}}+\bar{b}_{2}(\hh M)e^{c_2\hh \mathcal{D}}\big)^{-1},\\
\mathcal{L}_2(\hh \mathcal{D}):= &\big(e^{\hh \mathcal{D}} -e^{\hh M } \big)
\big(\bar{b}_{1}(\hh M)(e^{c_1\hh \mathcal{D}}+C  \hh^2\mathcal{D}^2)\\&+\bar{b}_{2}(\hh M)(e^{c_2\hh \mathcal{D}}+C  \hh^2\mathcal{D}^2)\big)^{-1}-\mathcal{L}_1(\hh \mathcal{D}),\\
\mathcal{L}(\hh \mathcal{D}):=&\mathcal{L}_1(\hh \mathcal{D})+\mathcal{L}_2(\hh \mathcal{D}),
\end{array}\end{equation*}
one has  $$
 \mathcal{L}(\hh \mathcal{D}) \Phi( t )=\hh    \Gamma ( t,\Phi( t )).$$
Expanding  the
nonlinearity into its Taylor series yields
\begin{equation*}
\begin{aligned}&\mathcal{L}(\hh \mathcal{D})\Phi( t )
=\hh  \sum\limits_{\mathbf{k}\in\mathcal{N}^*_{\infty}}\mathrm{e}^{\mathrm{i}(\mathbf{k}
\cdot \boldsymbol{\omega})  t } \sum\limits_{m\geq
2}\frac{\Gamma^{(m)}(t,0)}{m!}
\sum\limits_{\mathbf{k}^1+\ldots+\mathbf{k}^m=\mathbf{k}} \Big[\alpha^{\mathbf{k}^1}\cdot\ldots\cdot
\alpha^{\mathbf{k}^m}\Big]( t ).
\end{aligned} %
\end{equation*}
We remark that the assumptions $f(0)=0$ and  $f'(0)=0$ are used here.
Inserting  the ansatz \eqref{MFE-1} into these expressions and comparing the coefficients
of $\mathrm{e}^{\mathrm{i}(\mathbf{k} \cdot\boldsymbol{\omega} )  t }$, we obtain
\begin{equation*}
\begin{aligned}
&\mathcal{L}(\hh \mathcal{D}+\mathrm{i}(\mathbf{k} \cdot\boldsymbol{\omega} )
\hh)\alpha^{\mathbf{k}}( t )=\hh  \sum\limits_{m\geq
2}\frac{\Gamma^{(m)}(t,0)}{m!}
\sum\limits_{\mathbf{k}^1+\ldots+\mathbf{k}^m=\mathbf{k}} \Big[\alpha^{\mathbf{k}^1}\cdot\ldots\cdot
\alpha^{\mathbf{k}^m}\Big]( t ).
\end{aligned} %
\end{equation*}
This is the  modulation system  for the coefficients
$\alpha^{\mathbf{k}}( t )$  of the modulated Fourier expansion.


  In the light of the coefficients of S2O3,  the   Taylor expansions of $  \mathcal{L}_1(\hh \mathcal{D})$ are given by
\begin{equation}\label{LhD}
\begin{aligned}\mathcal{L}_1(\hh \mathcal{D})= &-\hh \Omega \mathrm{i}  +\hh \mathcal{D}-
 \frac{1}{2}\hh \Omega^{-1}\big(-2I+\hh \Omega \cot(\hh \Omega/2 )\big)\mathrm{i} \mathcal{D}^2+\ldots,\\
\mathcal{L}_1(\hh \mathcal{D}+\mathrm{i} \hh(\mathbf{k}\cdot \boldsymbol{\omega}))=&- \hh^2 \Omega^2  \Big(\hh  \Omega-
\cot\zeta+
\cos\zeta\csc\zeta\Big)^{-1}
\mathrm{i}+\ldots,
\end{aligned}
\end{equation}
with $\zeta:= \frac{1}{2}\hh (\Omega -(\mathbf{k}\cdot \boldsymbol{\omega})I) $.
Some   particular components we need are expressed as
\begin{equation}\label{LhDs}
\begin{aligned}
\big(\mathcal{L}_1(\hh \mathcal{D}+i \hh\langle j\rangle_l\cdot \boldsymbol{\omega})\big)_{j,l}=&
\frac{\hh^2 \omega_{j,l}^2}{4\sin^2(\hh \omega_{j,l}/2)}\hh \mathcal{D} \\&-\mathrm{i}\frac{\hh^2\omega_{j,l}^2}{16\sin^4(\hh \omega_{j,l}/2)} (\hh \omega_{j,l}-\sin(\hh \omega_{j,l})) (\hh \mathcal{D})^2+\ldots.\\
\end{aligned}
\end{equation}
Similarly, we get  $\mathcal{L}_2(\hh \mathcal{D})= -C\mathcal{D}^2(\mathcal{D}-\mathrm{i} \Omega) \hh^3+\cdots$ and this demonstrates that the main part of $\mathcal{L}(\hh \mathcal{D})$ comes from $\mathcal{L}_1(\hh \mathcal{D})$.
With these Taylor expansions and in  the spirit
of Euler's derivation of the Euler-Maclaurin summation formula (see  Chapter II. 10 of \cite{Wanner}),
    the following ansatz of the modulated
Fourier functions $\alpha^{\mathbf{k}}( t )$ is derived (see   \cite{WZ21} for the details of the derivation):
\begin{equation}\label{ansatz}%
\begin{array}{ll}
& \dot{\alpha}_{j,l}^{\langle j\rangle_l}( t )=\frac{\hh  }{ \iota_1}
\Gamma^{\mathbf{1}}_{j0}(\cdot)+\ldots,\ j=-\frac{N_{\tau}}{2},-\frac{N_{\tau}}{2}+1,\ldots,\frac{N_{\tau}}{2},\\
& \alpha^{\mathbf{k}}( t )=\frac{\hh  }{ \iota_2} \big(\Gamma^{\mathbf{k}}_{j0}(\cdot)+\ldots\big),\
\mathbf{k}\neq\langle j \rangle_l,
\end{array} %
\end{equation}
where  the dots  mean the power series in $\hh$ and $\Gamma^{\mathbf{k}}$ and so on stand for formal
series.   We
 truncate the
series after the $\mathcal{O}(\hh ^{N+N_0})$ terms for arbitrary positive integer $N_0$ since they often diverge.
The initial values for the differential equations appeared in the ansatz are
 determined by considering  $\Phi(0)=\widehat{Z^0}$. We thus get $\widehat{Z_{j,l}^0}=\alpha^{\langle j\rangle_l}_{j,l}(0)+\mathcal{O}(\hh    \delta_0^2),$ which yields  $\alpha^{\langle j\rangle_l}_{j,l}(0)=\mathcal{O}(\delta_0)$.

 Under the above analysis, the numerical result  obtained from
S2O3  can be expressed by   \eqref{MFE-ERKN-0}.
From the construction of the coefficient functions, it follows that it is reasonable to assume $\alpha^{\mathbf{k}}_{0,l}=0$ if $\mathbf{k}\neq \langle 0\rangle_l$ and   $\alpha^{\mathbf{k}}_{:,m}=0$ if $\abs{\mathbf{k}_{:,l}}>0\ \textmd{and}\ l\neq m$. Considering   the fact $\widehat{Z^{n}} \in \mathbb{R}^{D}$ yields that $\alpha^{-\mathbf{k}}_{-l,j}=\overline{\alpha^{\mathbf{k}}_{l,j}}$.

 \textbf{Proof of \eqref{coefficient func}.}
 In what follows, we derive the bounds of the coefficient functions
$\alpha^{\mathbf{k}}( t )$.
 By  the first formula of \eqref{ansatz}, we obtain a  coarse estimate as $$\dot{\alpha}_{j,l}^{\langle j\rangle_l}( t )=\mathcal{O}\Big(\frac{\hh  }{ \iota_1}\Big)\mathcal{O} (\Gamma^{\mathbf{1}})=\mathcal{O}\Big(\frac{\hh }{ \iota_1} \Big)$$ and by further considering the bound of $\alpha^{\langle j\rangle_l}_{j,l}(0)$,  it is arrived that $$\alpha_{j,l}^{\langle j\rangle_l}( t )=\mathcal{O}(\delta_0).$$
 Based on this estimate,  it is deduced that $\Gamma^{\mathbf{k}}=\mathcal{O}(\delta_0^{|\mathbf{k}|}).$ Therefore,
  $\alpha_{j,l}^{\mathbf{k}}( t )$  is bounded as $$ \alpha_{j,l}^{\mathbf{k}}( t )=\mathcal{O}\Big(\frac{\hh\varepsilon  }{ \norm{\iota_2}}   \delta_0^{|\mathbf{k}|} \Big)$$ by the second formula of \eqref{ansatz}. According to these estimates, a   finer bound of $\Gamma^{\mathbf{1}}$ is determined by
$$
\sum\limits_{\mathbf{k}^1+\cdots+\mathbf{k}^m=\langle j\rangle_l} \Big[\eta^{\mathbf{k}^1}_{j,l}\cdot\ldots\cdot
\eta^{\mathbf{k}^m}_{j,l}\Big],$$ which yields     $\Gamma^{\mathbf{1}}=\mathcal{O}(\frac{\hh }{  \iota_2}  \delta_0^3).$
Thence, we get $$\dot{\alpha}_{j,l}^{\langle j\rangle_l}( t )=\mathcal{O}(\frac{\hh^2}{  \iota_1\norm{\iota_2}}  \delta_0^3).$$  Therefore, the bounds presented in \eqref{coefficient func} are derived.

 \textbf{Proof of \eqref{remainder}.}
 Finally, we show the  bound of  remainder    \eqref{remainder}.
 First insert $\Phi( t )$  into the
numerical scheme \eqref{cs ei-ful2} and then the corresponding discrepancies are  (here we omit $(\hh M)$  for conciseness)
\begin{equation*}
\begin{array}[c]{ll}d_1( t )=&\Phi( t +c_1\hh )-e^{c_1\hh M}\Phi( t )\\&-\hh   \big(\bar{a}_{11}    \Gamma \big( t +c_1\hh ,\Phi( t +c_1\hh )\big)+\bar{a}_{12}    \Gamma \big( t +c_2\hh ,\Phi( t +c_2\hh )\big)\big),\\
d_2( t )=&\Phi( t +c_2\hh )-e^{c_2\hh M}\Phi( t )\\&-\hh   \big(\bar{a}_{21}    \Gamma \big( t +c_1\hh ,\Phi( t +c_1\hh )\big)+\bar{a}_{22} \Gamma \big( t +c_2\hh ,\Phi( t +c_2\hh )\big)\big),\\
d_3( t )=&\Phi( t +\hh)-e^{\hh M}\Phi( t )\\&-\hh   \big(\bar{b}_{1}  \Gamma \big( t +c_1\hh ,\Phi( t +c_1\hh )\big)+\bar{b}_{2} \Gamma \big(t +c_2\hh ,\Phi( t +c_2\hh )\big)\big).
\end{array}
\end{equation*}
There are two aspects in the bounds of these discrepancies:  $\mathcal{O}\Big(\frac{\hh   }{ \iota_2}   \delta_0^{N+1} \Big)=\mathcal{O}(  \eps   \delta_0^{N+1})$ in the truncation of  modulated Fourier expansions and $\mathcal{O}(\hh ^{N+N_0})$ in  the truncation of  the ansatz \eqref{ansatz}.
 Therefore,  discrepancies are bounded by
$$d_j( t )=\mathcal{O}(\hh ^{N+N_0})+\mathcal{O}(  \eps   \delta_0^{N+1})=\mathcal{O}(  \eps \delta_0^{N+1})$$ for $j=1,2,3$
on the basis of the arbitrarily large  $N_0$.
Then define the  errors  $$e^{Z }_n=\widehat{Z^{n}}-\Phi( t _n),\ E^{Z }_{ni}=\widehat{Z^{ni}}-\Phi( t _n+c_i\hh ).$$
They
satisfy the error recursion
\begin{equation*}
\begin{array}[c]{ll} E^{Z }_{ni}=&e^{c_i\hh M}e^{Z }_{n}+\hh\varepsilon
\textstyle\sum\limits_{\rho=1}^{2}\bar{a}_{i\rho}(\hh M)
 \big(\Gamma ( t _n+c_{\rho}\hh,\widehat{Z^{n\rho}}) \\&-\Gamma ( t _n+c_{\rho}\hh,\Phi( t _n+c_{\rho}\hh))\big)+d_i( t _n),\ i=1,2,\\
e^{Z }_{n+1}=&e^{\hh M}e^{Z }_{n}+\hh\varepsilon
\textstyle\sum\limits_{\rho=1}^{2}\bar{b}_{\rho}(\hh M)     \big(\Gamma ( t _n+c_{\rho}\hh,\widehat{Z^{n\rho}}) \\&-\Gamma ( t _n+c_{\rho}\hh,\Phi( t _n+c_{\rho}\hh ))\big)+d_3( t _n).\\
\end{array}
\end{equation*}
{Taking  the Lipschitz condition into account}, we
obtain $$\norm{ \Gamma ( t _n+c_{\rho}\hh,\widehat{Z^{n\rho}}) -\Gamma ( t _n+c_{\rho}\hh,\Phi( t _n+c_{\rho}\hh))}
\leq L \norm{E^{Z }_{n\rho}}.$$
The application of the Gronwall inequality now shows the boundedness of the defect \eqref{remainder}.
\end{proof}

\textbf{II. Almost-invariant.}
We have derived the modulated Fourier expansion  \eqref{MFE-ERKN-0}  of  S2O3
integrator and its long time energy   conservation  will be studied on the basis of an almost-invariant  of the
coefficient functions of   \eqref{MFE-ERKN-0}. Using the same arguments of \cite{WZ21}, the following almost-invariant can be derived.
\begin{lemma}\label{ai lem}
Define  the almost-invariant
$\mathcal{H}( t ):=\frac{1}{\eps^2}\mathcal{I}( t )+\mathcal{H}_1( t ),$ which satisfies  $$\eps^2 \mathcal{H}( t )=\eps^2
\mathcal{H}(0)+\mathcal{O}( t \eps^2 \hh^{2}\delta_0^{N+1}).$$
Here  $\mathcal{I}$ and $\mathcal{H}_1$ are expressed as
\begin{equation*}\begin{aligned}
&\mathcal{I}( t )
  = \sum\limits_{l=1}^{\hat{d}} \sum\limits_{j=-N_{\tau}/2}^{N_{\tau}/2} \Big(\Omega^2_{j,l}
  \abs{\alpha^{\langle j \rangle_l}_{j,l}}^2( t )+\abs{\alpha^{\langle j \rangle_{l+\hat{d}}}_{j,l+\hat{d}}}^2( t )\Big)+\mathcal{O}(\eps^{3}\delta_0^4),\\
&  \mathcal{H}_1( t )=\mathcal{V} ( \vec{\alpha}( t ) )+{\mathcal{O}( \eps^{2} \delta_0^4)},
\end{aligned}
\end{equation*}
with $ \vec{\alpha}( t ) =\big(\alpha^{\mathbf{k} }( t )\big)_{\mathbf{k}\in \mathcal{N}^*}$ and the potential $$\mathcal{V}(\vec{\alpha}( t ) ):=\sum\limits_{m= 1}^N
\frac{H_1^{(m+1)}(0)}{(m+1)!}
\sum\limits_{\mathbf{k}^1+\cdots+\mathbf{k}^{m+1}=\mathbf{0}}
 \left(\alpha^{\mathbf{k}^1}_{h}\cdot\ldots\cdot
\alpha^{\mathbf{k}^{m+1}}_{h}\right)( t ).$$
Here we use the   notation $\alpha^{\mathbf{k}}_{h}( t )=\mathrm{e}^{i (\mathbf{k} \cdot \boldsymbol{\omega})
t}\alpha^{\mathbf{k}}( t )$ and
   the potential $H_1$ given in \eqref{hhh}.
Moreover, the
relationship between this almost-invariant and  the result  $H$ of
the numerical method is derived as
\begin{equation*}
\begin{aligned}
\eps^2 \mathcal{H}( t _n)=\eps^2 H(u^{n},v^{n})+\mathcal{O}(\eps^{3} \delta_0^4)+\mathcal{O}(  \delta_{\mathcal{F}}).
\end{aligned}
\end{equation*}
\end{lemma}

With the above results, it is easy to get
\begin{equation*}
\begin{aligned}
 \eps^2 H(u^{n},v^{n})&=  \eps^2\mathcal{H}( t _n)+\mathcal{O}( \eps^{3} \delta_0^4)+ \mathcal{O}(  \delta_{\mathcal{F}}) \\
 &= \eps^2 \mathcal{H}( t _{0})+n\hh\mathcal{O}(\hh^2 \eps^{2}\delta_0^{N+1})+\mathcal{O}(\eps^{3} \delta_0^4)+  \mathcal{O}(  \delta_{\mathcal{F}})\\
  &=\ldots= \eps^2 H(u^{0},v^{0})+\mathcal{O}( \eps^{3} \delta_0^4)+ \mathcal{O}(  \delta_{\mathcal{F}})
\end{aligned}
\end{equation*}
as long as $n\hh^3\eps^{2}\delta_0^{N+1} \leq \eps^{3}\delta_0^4$.
This completes the proof of Theorem \ref{Long-time thm} for S2O3.

\textbf{III. Proof for S3O4.}
For the integrator S3O4,     it follows from its coefficients that
$$\widehat{Z^{nj}}=\Phi( t +c_j\hh )+C\eps^{5}\hh^4 \mathcal{D}^4\Phi( t +c_j\hh )
$$ for $j=1,2,3$
with the error constant $C$.
We define the
  operator     \begin{equation*} \begin{array}[c]{ll}\mathcal{L}_{S3}(\hh \mathcal{D}):= &\big(e^{\hh \mathcal{D}} -e^{\hh M } \big)
\big(\bar{b}_{1}(\hh M)(e^{c_1\hh \mathcal{D}}+C \hh^4 \mathcal{D}^4)\\&+\bar{b}_{2}(\hh M)(e^{c_1\hh \mathcal{D}}+C \hh^4 \mathcal{D}^4)+\bar{b}_{3}(\hh M)(e^{c_1\hh \mathcal{D}}+C \hh^4 \mathcal{D}^4)\big)^{-1}
\end{array}\end{equation*} for S3O4.
Then   the rest proceeds similarly to that stated
above for S2O3. It should be noted that the bounds of the coefficients in modulated expansion have the same expression as S2O3. 
For S3O4, the  notations $\iota_1, \iota_2$ appeared in the above proof \eqref{iota} become
$$\iota_1:=\frac{\hh^4 \omega_{j,l}^3}{ \sin^3(\hh \omega_{j,l}/4)(64\cos(\hh \omega_{j,l}/4)+16\hh\omega_{j,l}\sin(\hh \omega_{j,l}/4))},\
  \iota_2:=\mathcal{O}(\hh )   \Omega,$$
based on the Taylor expansion of $\mathcal{L}_{S3}$.
Therefore,  the energy at S3O4 has the following relation with the almost-invariant $\mathcal{H}$: $$
\eps^2 \mathcal{H}( t _n)=\eps^2 H(u^{n},v^{n})+\mathcal{O}(\eps^{3}\delta_0^4)+ \mathcal{O}(  \delta_{\mathcal{F}}).
$$
This result yields the estimate of  S3O4 given in Theorem \ref{Long-time thm}.
\end{proof}


\section{Conclusion}\label{sec:con}
In this paper, we have designed and analyzed two-scale integrators
for  the  nonlinear  Klein--Gordon  equation  with a dimensionless parameter  $0<\varepsilon\ll 1$.  Using
some transformations of the original system,  two-scale formulation
approach, spectral semi-discretisation and
exponential integrators with stiff order and symmetric conditions, a class of large-stepsize highly accurate integrators was formulated as the
numerical approximation of \eqref{klein-gordon} with large
initial data. Stiff order conditions and symmetric property were
used in the construction of practical methods.  The proposed
integrators were shown to have not only high accuracy
but also good long-term energy {near conservation}. The numerical results of a numerical
experiment supported the properties of the obtained integrators.

{Last but not least, we point out that the main
contribution of this paper is that we have established a new
framework to design  uniform higher-order integrators with  long time behavior
for solving highly oscillatory differential equations
with  strong nonlinearity.
We believe that the methodology presented in this paper can be
extended to a range of  nonlinear Hamiltonian PDEs such as the Dirac equation and Schr\"{o}dinger equation. The rigorous analysis on this topic will be considered in our next work. Another issue for future exploration is the study of uniform higher-order integrators with exact structure  conservation such as symplecticity and energy.
  }

\section*{Conflict of interest}
The authors declare that they have no known competing financial interests or personal
relationships that could have appeared to influence the work reported in this paper.

\section*{Acknowledgements}
The authors sincerely thank  the   anonymous reviewer  for
the very valuable comments and helpful suggestions.
This work was supported by NSFC (12371403, 12271426).

\bibliographystyle{model-num-names}

\begin{thebibliography}{00}
 \bibitem{Bao23}
{\sc W. Bao, Y. Cai,   and Y. Feng},  Improved uniform error bounds of the time-splitting methods for the long-time (nonlinear) Schr\"{o}dinger equation, Math. Comp.,  92 (2023),  1109-1139.


   \bibitem{Bao21}
{\sc W. Bao, Y. Cai,   and Y. Feng}, Improved uniform error bounds on time-splitting methods for long-time dynamics of the nonlinear Klein-Gordon equation with weak nonlinearity, SIAM J. Numer. Anal.,  60 (2022),  1962-1984.


 \bibitem{Baocai}
{\sc W. Bao, Y. Cai, X. Jia, and Q. Tang}, A uniformly accurate multiscale time integrator pseudospectral method for the Dirac equation in the nonrelativistic limit regime, SIAM J. Numer. Anal. 54 (2016)  1785-1812.


\bibitem{bao14} {\sc W. Bao, Y. Cai, and X. Zhao}, {A uniformly accurate multiscale time integrator pseudospectral method for the Klein-Gordon equation in the nonrelativistic limit regime}, SIAM J. Numer. Anal., 52 (2014)  2488-2511.

    \bibitem{zhao19} {\sc W. Bao and X. Zhao}, {Comparison of numerical methods for the nonlinear Klein--Gordon equation
in the nonrelativistic limit regime},  J. Comput. Phys.,
398 (2019),  108886.

\bibitem{PI1}
{\sc S. Baumstark, E. Faou, and K. Schratz}, Uniformly accurate exponential-type integrators for Klein-Gordon equations with asymptotic convergence to classical splitting schemes in the NLS splitting, Math. Comp. 87 (2018)  1227-1254.

\bibitem{S2}{\sc S. Baumstark and K. Schratz}, Uniformly accurate  oscillatory integrators for the Klein-Gordon-Zakharov systems from low-to high-plasma frequency regimes, SIAM J. Numer. Anal.,  57  (2019),  429-457.

     \bibitem{r4}{\sc L. Brugnano, F. Iavernaro, and D. Trigiante}, {Energy and quadratic invariants preserving integrators based upon Gauss collocation formulae}, SIAM J. Numer. Anal.  50 (2012)  2897-2916.

 \bibitem{B19}{\sc L.Brugnano, J. I. Montijano, and L. R\'{a}ndez}, On the effectiveness of spectral methods for the numerical solution of multi-frequency highly-oscillatory Hamiltonian problems, Numer. Algo. 81 (2019)   345-376.



     \bibitem{bb}{\sc J. Butcher}, {B-Series: Algebraic Analysis of Numerical Methods},   Springer, Cham,
   2021.

\bibitem{S1} {\sc M. Cabrera Calvo and K. Schratz}. Uniformly accurate low regularity integrators for the Klein-Gordon equation
from the classical to non-relativistic limit regime,
SIAM J. Numer. Anal.    60 (2022)  888-912.

\bibitem{PI2}
{\sc Y. Cai and  Y. Wang}, Uniformly accurate nested Picard iterative integrators for the Dirac equation in the nonrelativistic limit regime, SIAM J. Numer. Anal. 57 (2019)  1602-1624.



\bibitem{Chartier15}  {\sc Ph. Chartier, N. Crouseilles, M. Lemou, and F. M\'{e}hats}, {Uniformly accurate numerical schemes for highly oscillatory Klein-Gordon and nonlinear Schr\"{o}dinger equations}, Numer. Math. 129 (2015)   211-250.

    \bibitem{vp3D}
{\sc Ph. Chartier, N. Crouseilles, M. Lemou, F. M\'{e}hats, and X. Zhao}, Uniformly accurate methods for three dimensional Vlasov equations under strong magnetic field with varying direction, SIAM J. Sci. Comput. 42 (2020)  B520-B547.

 \bibitem{2scale multi}
{\sc Ph. Chartier, M. Lemou, and  F. M\'{e}hats},
Highly-oscillatory evolution equations with multiple frequencies: averaging and numerics,
Numer. Math. 136 (2017)  907-939.


\bibitem{NUA}
{\sc Ph. Chartier, M. Lemou, F. M\'{e}hats, and G. Vilmart},
{A new class of uniformly accurate methods for highly oscillatory evolution equations},
Found. Comput. Math.  20 (2020)  1-33.

\bibitem{autoUA}  {\sc Ph. Chartier, M. Lemou, F. M\'{e}hats,   and X. Zhao},
 Derivative-free high-order uniformly accurate schemes for highly-oscillatory systems, IMA J. Numer. Anal.,  42 (2022)  1623-1644.

\bibitem{Cohen0}
{\sc D. Cohen and  L. Gauckler},
One-stage exponential integrators for nonlinear
Schr\"odinger equations over long times, BIT Numer. Math. 52 (2012)  877-903.

\bibitem{Cohen2}
{\sc D. Cohen, E. Hairer,  and Ch. Lubich}, Long-time analysis of nonlinearly perturbed wave equations via modulated Fourier expansions, Arch. Ration. Mech. Anal.
187 (2008)  341-368.

\bibitem{Lubich2006}
{\sc D. Cohen, T. Jahnke, K. Lorenz,   and Ch. Lubich},
  Numerical integrators for highly oscillatory Hamiltonian
systems: a review,    in Analysis, Modeling and Simulation of
Multiscale Problems (A. Mielke, ed.), Springer, Berlin, (2006)
 553-576.

\bibitem {Feng-book}{\sc K. Feng and M. Qin},  {Symplectic Geometric Algorithms for Hamiltonian
Systems}, Springer-Verlag, Berlin, Heidelberg, 2010.

\bibitem {franco2006}
{\sc J.M. Franco},
  New methods for oscillatory systems
based on ARKN methods,    Appl. Numer. Math.   56  (2006)
1040-1053.

\bibitem{MD}
{\sc B. Garc\'{i}a-Archilla, J.M. Sanz-Serna,  and R.D. Skeel}, Long-time-step methods
for oscillatory differential equations, SIAM J. Sci. Comput. 20 (1999)  930-963.




\bibitem{Lubich ICM}
{\sc L. Gauckler, E. Hairer, and  Ch. Lubich}, Dynamics, numerical analysis, and some geometry, Proc. Int. Cong. Math. 1 (2018)  453-486.



\bibitem{Lubich2}
{\sc L. Gauckler and  Ch. Lubich}, Splitting integrators for
nonlinear Schr\"{o}dinger equations over long times, Found. Comput. Math. 10 (2010)  275-302.



\bibitem{Grimm}
{\sc V. Grimm}, On error bounds for the Gautschi-type exponential integrator applied to oscillatory second-order differential equations, Numer. Math. 100 (2005)  71-89.

\bibitem{hairer2000}
{\sc  E. Hairer and  Ch. Lubich}, {Long-time energy
conservation of numerical methods for oscillatory differential
equations}, SIAM J. Numer. Anal.,   38 (2000)  414-441.

\bibitem{hairer2006}
{\sc E. Hairer, Ch. Lubich,  and G. Wanner},  {Geometric Numerical Integration: Structure-Preserving Algorithms for Ordinary
Differential Equations},    2nd ed., Springer-Verlag, Berlin,
Heidelberg,  2006.

  \bibitem{lubich19}{\sc E. Hairer, Ch. Lubich, and  B. Wang}, {A filtered Boris algorithm for
charged-particle dynamics in a strong magnetic field},
Numer. Math. 144 (2020)   787-809.
{ \bibitem{Wanner}
{\sc E. Hairer,  G. Wanner}, Analysis by its history, 2nd printing, Undergraduate Texts in Mathematics, Springer-Verlag, New York, 1997.}

\bibitem{Hochbruck1999}
{\sc M.  Hochbruck and Ch. Lubich},   {A Gautschi-type method for oscillatory second-order differential equations}, Numer. Math.  83
(1999)  403-426.



\bibitem{Ostermann06}
{\sc M. Hochbruck and  A. Ostermann},  Explicit exponential Runge--Kutta methods for semilinear parabolic problems, SIAM J. Numer. Anal. 43 (2006)  1069-1090.


\bibitem{Ostermann}
{\sc M. Hochbruck and  A. Ostermann}, Exponential integrators, Acta Numer. 19 (2010)  209-286.

\bibitem {Iserles02} {\sc A. Iserles},  On the global error of discretization methods for highly-oscillatory ordinary differential equations, BIT 42 (2002)     561–599.




%
%

\bibitem{zhao18} {\sc N.J. Mauser, Y. Zhang, and X. Zhao}, {On the rotating nonlinear Klein-Gordon equation: non-relativistic limit and
numerical methods}, SIAM J. Multi. Model. Simu.  (2020)   999-1024.


  \bibitem {Sanz-Serna08}   {\sc J. M. Sanz-Serna}, Mollified impulse methods for highly-oscillatory differential equations, SIAM J. Numer. Anal. 46 (2008)     1040-1059.

      \bibitem {Sanz-Serna}  {\sc J.M. Sanz-Serna}, Symplectic Runge-Kutta schemes for adjoint equations, automatic differentiation, optimal control and more, SIAM Rev. 58 (2016)     3-33.

  \bibitem {Schratz}  {\sc K. Schratz and X. Zhao}, On the comparison of the asymptotic expansion techniques for the nonlinear Klein-Gordon
equation in the non relativistic limit regime.
DCDS-B  25 (2020) 2841-2865.



\bibitem{Shen}
{\sc J. Shen, T. Tang,  and L. Wang},
Spectral Methods: Algorithms, Analysis and Applications, Springer, 2011.

        \bibitem{Shen19} {\sc J. Shen, J. Xu, and J. Yang}, {A new class of efficient and robust energy stable schemes
for gradient flows}, SIAM Rev.  61 (2019)   474-506.

\bibitem{WIW}
{\sc B. Wang, A. Iserles,  and X. Wu}, Arbitrary-order trigonometric Fourier collocation methods for multi-frequency oscillatory systems, Found. Comput. Math. 16 (2016)  151-181


\bibitem{WW}
{\sc B. Wang and  X. Wu},  A long-term numerical energy-preserving analysis of symmetric and/or symplectic extended RKN integrators for efficiently solving highly oscillatory Hamiltonian systems,  BIT Numer. Math. 61 (2021)  977-1004.

\bibitem{WZ}
{\sc B. Wang and  X. Zhao}, Error estimates of some splitting schemes for
charged-particle dynamics under strong magnetic field, SIAM J. Numer. Anal. 59 (2021)  2075-2105.

\bibitem{WZ21}
{\sc B. Wang and  X. Zhao}, Geometric two-scale integrators for highly oscillatory system: uniform accuracy and near conservations, SIAM J. Numer. Anal.  61 (2023) 1246-1277.

\bibitem{Wu}
{\sc X. Wu and B. Wang}, Geometric Integrators for Differential Equations with Highly Oscillatory Solutions, Springer, Singapore,  2021.




\bibitem{Zhao}
{\sc X. Zhao},
Uniformly accurate multiscale time integrators for second order oscillatory differential equations with large initial data, BIT Numer. Math. 57 (2017)  649-683.





\end{thebibliography}

\end{document}